\documentclass[a4paper, 12pt]{article}

\usepackage{amsfonts}
\usepackage {amssymb}
\usepackage {amsmath}
\usepackage {amsmath}
\usepackage {amsthm}
\usepackage{graphicx}
\usepackage{multirow}
\usepackage{xypic}
\usepackage {amscd}
\usepackage{mathrsfs}
\usepackage[colorlinks, linkcolor=blue, anchorcolor=black, citecolor=red]{hyperref}
\usepackage{enumerate}

\usepackage{geometry}  

\newcommand{\Fut}{{\rm Fut}}

\newcommand{\mcY}{{\mathcal{Y}}}
\newcommand{\mcL}{{\mathcal{L}}}
\newcommand{\mcX}{{\mathcal{X}}}

\newcommand{\mcM}{{\mathcal{M}}}

\newcommand{\ord}{{\rm ord}}
\newcommand{\lct}{{\rm lct}}
\newcommand{\vphi}{\varphi}

\newcommand{\bC}{{\mathbb{C}}}

\newcommand{\FS}{{\rm FS}}
\newcommand{\bP}{{\mathbb{P}}}

\newcommand{\fb}{\mathfrak{b}}
\newcommand{\mcW}{\mathcal{W}}
\newcommand{\chw}{{\rm CW}}

\newcommand{\mcQ}{{\mathcal{Q}}}

\newcommand{\sslash}{{/\!/}}
\newcommand{\bG}{\mathbb{G}}
\newcommand{\bK}{\mathbb{K}}
\newcommand{\bT}{\mathbb{T}}

\newcommand{\NA}{{\rm NA}}
\newcommand{\MA}{{\rm MA}}

\newcommand{\mcJ}{{\mathcal{J}}}
\newcommand{\mcZ}{\mathcal{Z}}

\newcommand{\triv}{{\rm triv}}
\newcommand{\bQ}{{\mathbb{Q}}}

\newcommand{\bR}{{\mathbb{R}}}

\newcommand{\mcH}{{\mathcal{H}}}
\newcommand{\mcO}{{\mathcal{O}}}
\newcommand{\mcF}{{\mathcal{F}}}
\newcommand{\bZ}{{\mathbb{Z}}}

\newcommand{\red}{{\rm red}}

\newcommand{\bN}{{\mathbb{N}}}

\newcommand{\la}{\langle}
\newcommand{\ra}{\rangle}
\newcommand{\mcE}{\mathcal{E}}

\newcommand{\mcI}{\mathcal{I}}

\newcommand{\mcU}{{\mathcal{U}}}

\newcommand{\bfM}{{\bf M}}
\newcommand{\bfH}{{\bf H}}
\newcommand{\bfE}{{\bf E}}
\newcommand{\bfJ}{{\bf J}}

\newcommand{\bfF}{{\bf F}}

\newcommand{\Aut}{{\rm Aut}}
\newcommand{\PSH}{{\rm PSH}}

\newcommand{\Lam}{{\bf \Lambda}}
\newcommand{\cF}{\mathcal{F}}
\newcommand{\bfI}{{\bf I}}
\newcommand{\udS}{\underline{S}}
\newcommand{\bB}{\mathbb{B}}
\newcommand{\ud}{\underline}

\newcommand{\bfR}{{\bf R}}

\newcommand{\fMO}{\mathfrak{MO}}
\newcommand{\fDM}{\mathfrak{DMO}}

\newcommand{\cE}{\mathcal{E}}

\newcommand{\fM}{\mathfrak{M}}

\newcommand{\bfG}{{\bf G}}

\newcommand{\cJ}{\mathcal{J}}
\newcommand{\bfS}{{\bf S}}

\newcommand{\cQ}{\mathcal{Q}}

\newcommand{\fSN}{\mathfrak{SN}}
\newcommand{\fDS}{\mathfrak{DSN}}
\newcommand{\ddc}{{\rm dd^c}}
\newcommand{\qV}{}

\newtheorem{thm}{Theorem}[section]
\newtheorem{prop}[thm]{Proposition}
\newtheorem{defn}[thm]{Definition}

\newtheorem{cor}[thm]{Corollary}
\newtheorem{rem}[thm]{Remark}
\newtheorem{conj}[thm]{Conjecture}

\newtheorem{exmp}[thm]{Example}
\newtheorem{lem}[thm]{Lemma}

\newtheorem{defn-prop}[thm]{Definition-Proposition}

\begin{document}

\title{Geodesic rays and stability in the cscK problem}
\author{Chi Li}
\date{}

\maketitle

\abstract{
We prove that any finite energy geodesic ray with a finite Mabuchi slope is maximal in the sense of Berman-Boucksom-Jonsson, and reduce the proof of the uniform Yau-Tian-Donaldson conjecture for constant scalar curvature K\"{a}hler metrics to Boucksom-Jonsson's regularization conjecture about the convergence of non-Archimedean entropy functional. As further applications, we show that a uniform K-stability condition for model filtrations and the $\mathcal{J}^{K_X}$-stability are both sufficient conditions for the existence of cscK metrics. The first condition is also conjectured to be necessary. Our arguments also produce a different proof of the toric uniform version of YTD conjecture for all polarized toric manifolds. 
Another result proved here is that the Mabuchi slope of a geodesic ray associated to a test configuration is equal to the non-Archimedean Mabuchi invariant. 
}

\tableofcontents

\section{Introduction and main results}

Let $(X, L)$ be a polarized algebraic manifold and $[\omega]=c_1(L)>0$ be the Hodge class. The Yau-Tian-Donaldson (YTD) conjecture aims to give a sufficient and necessary algebraic condition for the existence of constant scalar curvature K\"{a}hler (cscK) metric in the K\"{a}hler class $[\omega]$. Recently there have been significant progresses towards this conjecture, especially on the analytic part (see \cite{BDL17, BDL18, CC17,CC18a,CC18b, DL19}) and the Fano case (\cite{CDS15, Tia15, DaS16}). On the other hand, Berman-Boucksom-Jonsson \cite{BBJ15, BBJ18, Bo18} proposed a variational approach for attacking this conjecture, which has been successfully carried out in the Fano case, even for singular Fano varieties (see \cite{BBJ18, His16b, Li19, LTW19} and section \ref{sec-YTD2}).

Geodesic rays play important roles in the recent study of cscK problem. Indeed, it has been shown that the non-existence of cscK metric is equivalent to the existence of non-trivial destabilizing geodesic rays (see \cite{CC18b, BBJ18, DL19}). 
We recall the following definition and refer to section \ref{sec-rays} for definition of geodesic rays and \eqref{eq-Mabuchi} for the expression of Mabuchi energy $\bfM$.
\begin{defn}
For a finite energy geodesic ray $\Phi=\{\vphi(s)\}: \bR_{\ge 0}\rightarrow \cE^1(L)$, its Mabuchi slope is:
\begin{equation}
\bfM'^\infty(\Phi):=\lim_{s\rightarrow+\infty}\frac{\bfM(\vphi(s))}{s}.
\end{equation}
We say that $\Phi$ is destabilizing if $\bfM'^\infty(\Phi)\le 0$.
\end{defn}

The existence of the above limit, which may be $+\infty$, follows from the convexity of $\bfM$ along $\Phi$ as proved in \cite{BDL17} based on the convexity along $C^{1,\bar{1}}$-geodesics proved earlier in \cite{BB17}. 

To make contact with the YTD conjecture, we would like to know whether destabilizing geodesic rays are algebraically approximable, meaning that whether they can be approximated by a decreasing sequence of geodesic rays associated to test configurations. Such geodesic rays are maximal in the sense of Berman-Boucksom-Jonsson (\cite{BBJ18}). 
It was shown by Berman-Boucksom-Jonsson (\cite{BBJ18}) that there is a one-to-one correspondence between finite energy non-Archimedean metrics and maximal geodesic rays (see Theorem \ref{thm-BBJ}). 

Our first result says that any destabilizing geodesic rays are automatically algebraically approximable from above, i.e. maximal.
\begin{thm}\label{thm-maximal}
Let $\Phi: \bR_{\ge 0}\rightarrow \cE^1(L)$ be a geodesic ray. Assume that $\bfM'^\infty(\Phi)<\infty$. Then $\Phi$ is maximal. As a consequence, we have the identity:
\begin{equation}
\bfE'^\infty(\Phi)=\bfE^\NA(\Phi_\NA).
\end{equation}
\end{thm}

To solve the YTD conjecture it remains to show that Mabuchi slopes of destabilizing (maximal) geodesic rays are algebraically approximable, which means that they can be approximated by Mabuchi slopes of geodesic rays associated to test configurations. By Chen-Tian's formula in \eqref{eq-Mabuchi}, the Mabuchi energy has a decomposition into the entropy part and the energy part, and the energy part can be decomposed into the sum of Monge-Amp\`{e}re and twisted Monge-Amp\`{e}re energy. The slopes of Monge-Amp\`{e}re energy for maximal geodesics can be algebraically approximated (see Theorem \ref{thm-BBJ}).  
The following result shows that the twisted Monge-Amp\`{e}re slopes of maximal geodesic rays are also algebraically approximable. In this paper $\mcH^\NA(L)$ always denotes the set of 
(smooth positive) non-Archimedean metrics that are associated to semiample test configurations of $(X, L)$. 
\begin{thm}\label{thm-Echislope}
Let $(Q, \psi_Q)$ be a line bundle over $X$ with a smooth Hermitian metric $e^{-\psi_Q}$. 
Assume that $\Phi: \bR_{\ge 0}\rightarrow \cE^1(L)$ is a maximal geodesic ray. 
If $\{\phi_m\}\subset\mcH^\NA(L)$ is any sequence that converges strongly to $\Phi_\NA$, and $\Phi_m$ is the (maximal) geodesic ray associated to $\phi_m$, then we have the convergence:
\begin{equation}\label{eq-convEchiNA}
(\bfE^{\ddc \psi_Q} )'^\infty(\Phi)=\lim_{m\rightarrow+\infty}(\bfE^{\ddc \psi_Q})'^\infty(\Phi_m).
\end{equation}
As a consequence 
we always have the following identity for any maximal geodesic ray (see \eqref{eq-EcQNA}):
\begin{equation}\label{eq-EQslope}
(\bfE^{\ddc \psi_Q})'^\infty(\Phi)=(\bfE^{Q_\bC})^\NA(\Phi_\NA).
\end{equation}
\end{thm}

\begin{rem}\label{rem-strong}
Boucksom and Jonsson told me that the result in Theorem \ref{thm-Echislope} is independently known to them. Indeed, our proof of this result depends on some estimates that first appeared in \cite{BBGZ13} and later refined in \cite{BBEGZ, BBJ18}.
\end{rem}

In fact, a more general result is true (see Theorem \ref{thm-EcQslopeconv}), which will be used in our study of YTD conjecture in section \ref{sec-YTD}.
The above two results, combined with the recent progresses mentioned above, essentially reduce the proof of (the $\bG$-uniform version of) YTD conjecture for cscK metrics to proving that the entropy slope is algebraically approximable. 
More precisely 
it now suffices to prove the following result:
\begin{conj}\label{conj-entropy} 
Let $\Phi$ be a maximal geodesic ray. Then there is a sequence $\phi_{m}\in \mcH^{\NA}(L)$ converging to $\Phi_\NA$ in the strong topology such that (see \eqref{eq-HNAH} for $\bfH^\NA$)
\begin{equation}\label{eq-conjH}
\bfH'^\infty(\Phi)\ge 
\lim_{m\rightarrow+\infty} \bfH^\NA(\phi_m).
\end{equation}
\end{conj}

One difficulty in proving \eqref{eq-conjH} is that there is not yet an explicit formula for $\bfH'^\infty$ for general maximal geodesic rays.
On the other hand, for any finite energy $\phi\in \mcE^{1,\NA}(L)$, Boucksom-Jonsson in \cite{BoJ18b} defined the following non-Archimedean entropy in their study of K-stability (see \eqref{eq-AXsup}-\eqref{eq-HNAphi1}):
\begin{equation}
\bfH^\NA(\phi)=\qV \int_{X^\NA} A_X(x)\MA^\NA(\phi)(x).
\end{equation} 
Here we conjecture
\begin{conj}\label{conj-Hslope}
For any maximal geodesic ray $\Phi$, $\bfH'^\infty(\Phi)=\bfH^\NA(\Phi_\NA)$.
\end{conj}
This is partially answered by the following results:
\begin{thm}\label{thm-grslope}
\begin{enumerate}
\item For any geodesic ray $\Phi$, we always have the inequality:
\begin{equation}\label{eq-slopebigger}
\bfH'^\infty(\Phi)\ge \bfH^\NA(\Phi_\NA).
\end{equation}
\item
Given an ample test configuration $\pi: (\mcX, \mcL)\rightarrow \bC$, let $\Phi=\{\vphi(s)\}$ be a geodesic ray associated to $(\mcX, \mcL)$. Then we have the following slope formula:
\begin{eqnarray}
\bfH'^\infty(\Phi)&=&\bfH^\NA(\mcX, \mcL)=\qV K^{\log}_{\bar{\mcX}/X_{\bP^1}}\cdot \bar{\mcL}^n, \nonumber \\
\bfM'^\infty(\Phi)&=&\bfM^\NA(\mcX, \mcL)=\qV  K_{\bar{\mcX}/\bP^1}^{\rm log}\cdot \bar{\mcL}^n+\qV \frac{\ud{S}}{n+1}\bar{\mcL}^{n+1}. \label{eq-Mslope}
\end{eqnarray}
\end{enumerate}
\end{thm}
For the second statement, when $\Phi$ is replaced by a smooth $S^1$-invariant subgeodesic ray induced by a smooth positively curved metric on $\mcL$, the identity \eqref{eq-Mslope} was proved by Boucksom-Hisamoto-Jonsson (\cite{BHJ19}) and Tian (\cite{Tia97}, \cite[Lemma 2.1, 2.2]{Tia14}). 

Our results also highlight the importance of the following conjecture due to Boucksom-Jonsson.
\begin{conj}[Regularization Conjecture, \cite{BoJ18b}]\label{conj-reg}
For any $\phi\in \cE^{1,\NA}(L)$, there exists a sequence $\{\phi_m\}\subset \mcH^\NA(L)$ converging to $\phi$ in the strong topology such that:
\begin{equation}\label{eq-intAconv}
\int_{X^\NA}A_X(v)\MA^\NA(\phi)(v)=
\lim_{m\rightarrow+\infty} \int_{X^\NA}A_X(v)\MA^\NA(\phi_m)(v).
\end{equation}
\end{conj}
Indeed, the results proved here 
can show that the conjecture \ref{conj-reg} implies both Conjecture \ref{conj-entropy} and Conjecture \ref{conj-Hslope}, and hence also the uniform version of YTD conjecture (see Lemma \ref{lem-conjrel} and Proposition \ref{prop-YTDconj}). Although Conjecture \ref{conj-reg} is not known in general yet, we can nevertheless use Boucksom-Jonsson's non-Archimedean approach to K-stability 
(\cite{BoJ18a, BoJ18b}) and Boucksom-Favre-Jonsson's foundational work on non-Archimedean Monge-Amp\`{e}re equations (\cite{BFJ15, BoJ18a, BoJ18b}) to get an existence result involving K-stability for model filtrations. 


To state a general existence statement, let $\bG$ be a reductive complex Lie group in $\Aut(X, L)_0$ with a maximal compact subgroup $\bK$ such that $\bK^\bC=\bG$.
Let $\bT=((S^1)^r)^\bC$ be the identity component of the center of $\bG$. 
Set $N_\bZ={\rm Hom}(\bC^*, \bT)$ and $N_\bR=N_\bZ\otimes_\bZ \bR$. 
\begin{defn}
$(X, L)$ is called $\bG$-uniformly K-stable for models if there exists $\gamma>0$ such that for any normal model  $(\mcX, \mcL)$ of $(X, L)$, we have (see Definition \ref{def-model} and Definition \ref{def-modfil})
\begin{equation}\label{eq-Mmodel}
\bfM^\NA(\phi_{(\mcX, \mcL)}) \ge \gamma \inf_{\xi\in N_\bR} \bfJ^\NA(\phi_{(\mcX, \mcL),\xi}).
\end{equation}
\end{defn}
The K-stability notion in the above definition strengthens the usual definition of K-stability via test configurations. Our main existence result is:
\begin{thm}\label{thm-existmodel}
If $(X, L)$ is $\bG$-uniformly $K$-stable for models, then the Mabuchi functional of $(X, L)$ is proper and hence there is a cscK metric. 
\end{thm}
The converse direction, which we conjecture to be true when $\bG$ contains a maximal torus of $\Aut(X, L)_0$, would be implied by Conjecture \ref{conj-reg} for model filtrations.
 See Lemma \ref{lem-HisSze}, Proposition \ref{prop-YTDconj} and discussions there.
 
 \begin{rem}
In a following paper \cite{Li20b}, we will derive a movable intersection formula for the left-hand-side of \eqref{eq-Mmodel} which generalizes the formula \eqref{eq-MNAH} for test configurations:
\begin{equation}
\mcM^\NA(\phi_{(\mcX, \mcL)})=\la \bar{\mcL}_c^n\ra\cdot \left(K_{\bar{\mcX}/\bP^1}+\frac{\ud{S}}{n+1} \bar{\mcL}_c\right)
\end{equation}
where $\mcL_c=\mcL+c \mcX_0$ for $c\gg 1$ and $\la \bar{\mcL}_c^n\ra$ is the movable intersection product introduced in \cite{BDPP13}.
\end{rem}


For polarized toric manifolds, using the fact that toric model filtrations are the same as filtrations from toric test configurations, we see that Conjecture \ref{conj-reg} is true for toric model filtrations. As a consequence, we get another proof of the following result:
\begin{thm}[see \cite{His16b, CC18b}]\label{thm-toric}
Let $(X, L)$ be an $n$-dimensional polarized toric manifold. If $(X, L)$ is $(\bC^*)^n$-uniformly K-stable, then the Mabuchi functional of $(X, L)$ is proper, and hence there is a cscK metric. 
\end{thm}
This result was known by combining the works of Zhou-Zhu \cite{ZZ08}, Chen-Li-Sheng \cite{CLS14}, Hisamoto \cite{His16b} and Chen-Cheng \cite{CC18a, CC18b}. This previous proof depends on Donaldson's deep analysis of K\"{a}hler geometry on toric manifolds (\cite{Don02}). 
On the other hand, the proof given combines our analysis with the fundamental works on non-Archimedean Monge-Amp\`{e}re equations in \cite{BFJ15, BoJ18a}, in addition to \cite{CC18a, CC18b}.

Our study of variational approach to YTD also leads to the following sufficient algebraic criterion for the existence of cscK metric, which is related to the study of the so-called $J$-equation in the literature (see e.g. \cite{CC18a, SW08}). The proof of this result is independent of the above conjectures and
depends on Theorem \ref{thm-maximal}, Theorem \ref{thm-Echislope} and the variational argument. See Theorem \ref{thm-cJstable2} and discussions there. 
\begin{thm}\label{thm-cJstable}
If $(X, L)$ is $\cJ^{K_X}$-semistable, then $X$ admits a cscK metric.
\end{thm}

In the next section, we recall the non-Archimedean and Archimedean functionals, and explain they are related to each other. We state the basic correspondence between finite energy non-Archimedean metrics and maximal geodesic rays as established by Berman-Boucksom-Jonsson. We discuss the twist of non-Archimedean metrics that was introduced in \cite{His16b} for smooth ones and more generally developed in \cite{Li19}. We prove Theorem \ref{thm-maximal} in section \ref{sec-maximal} and Theorem \ref{thm-Echislope} in section \ref{sec-twistMA}. 
We prove Theorem \ref{thm-grslope} in section \ref{sec-YTD1}. 
In section \ref{sec-YTD2}, we prove the existence result Theorem \ref{thm-existmodel} (=Theorem  Theorem \ref{thm-existfiltr}) for cscK metric using $\bG$-uniform K-stability for model filtrations. This also allows us to prove the toric case
Theorem \ref{thm-toric}. The results about $\mcJ^{K_X}$-stability such as Theorem \ref{thm-cJstable} are dealt with in section \ref{sec-Jst}.

{\bf Acknowledgement:} 
The author is partially supported by NSF (Grant No. DMS-1810867) and an Alfred P. Sloan research fellowship. 
I am grateful to Sebastien Boucksom and Mattias Jonsson for many clarifying comments. 
In particular, the improvement of the presentation of the work here owes much to Boucksom's careful suggestions which also motivates me to incorporate some fundamental results from \cite{BFJ16, BoJ18a, BoJ18b}.
I would like to thank Chenyang Xu and Yuji Odaka for communications regarding Remark \ref{rem-XuOda}, Yuji Odaka for Remark \ref{rem-spherical} and bringing the paper \cite{Oda12} to his attention, Ruadha\'{i} Dervan for pointing out the reference \cite{DK19} and Mingchen Xia for helpful comments. The results in this paper was discussed at a work shop in American Institute of Mathematics (AIM) at the beginning of 2020, and the author would like to thank the organizers (Kento Fujita, Mattias Jonsson and Chenyang Xu) and AIM for creating a good environment of discussion. I also thank Yves de Cornulier for translating the title and abstract into French.

\section{Preliminaries}

\subsection{Non-Archimedean theory}\label{sec-NA}
\subsubsection{Non-Archimedean metrics and filtrations}
To fix the notations, we first recall some definitions. For more details, we refer to \cite{BoJ18a, BoJ18b}.
\begin{defn}\label{def-model}
\begin{enumerate}
\item
A model of $(X, L)$ is a flat family of projective varieties $\pi_\mcX: \mcX\rightarrow \bC$ together with a $\bC^*$-equivariant $\bQ$-line bundle $\mcL$ satisfying:
\begin{enumerate}
\item[(i)] There is a $\bC^*$-action on $(\mcX, \mcL)$ such that $\pi_\mcX$ is $\bC^*$-equivariant;
\item[(ii)] There is a $\bC^*$-equivariant isomorphism $(\mcX, \mcL)\times_{\bC}\bC^* \cong (X, L)\times\bC^*$.
\end{enumerate}
The trivial model of $(X, L)$ is given by $(X\times\bC, L\times\bC)=:(X_\bC, L_\bC)$.

Two models $(\mcX_i, \mcL_i), i=1,2$ are called equivalent if there exists a model $(\mcX_3, \mcL_3)$ and two $\bC^*$-equivariant birational morphisms $\mu_i: \mcX_3\rightarrow \mcX_i$ such that $\mu_1^*\mcL_1=\mu_2^*\mcL_2$.
\item
If we forget about the data $L$ and $\mcL$, then we say that $\mcX$ is a model of $X$.
We will denote by $\fMO$ the set of all models of $X$. 

If there is a $\bC^*$-equivariant birational morphism $r_{\mcX_1, \mcX_2}: \mcX_1\rightarrow \mcX_2$ of two models $\mcX_i, i=1,2\in \fMO$, then we say that $\mcX_1$ dominates $\mcX_2$ and write $\mcX_1\ge \mcX_2$. If $\mcX\ge X_\bC$, then we say that $\mcX$ is dominating, and we denote by $\fDM$ the set of all dominating models of $X$.

We say a model $\mcX\in \fMO$ is a SNC (i.e. simple normal crossing) model of $X$ if $(\mcX, \mcX^\red_0)$ is a simple normal crossing pair. We denote by $\fSN$ the set of all SNC models, and by $\fDS$ the set of dominating SNC models of $X$.

For any two models $\mcX_i, i=1,2$ of $X$, a common refinement of $\mcX_i, i=1,2$ is a model $\mcX_3$ such that $\mcX_3\ge  \mcX_i, i=1,2$. 
\end{enumerate}
\end{defn}
To any SNC model $\mcX$ of $X$ is associated a dual complex $\Delta_\mcX$, which is a simplicial complex whoses simplices are 1-1 correspondence with strata of $\mcX_0$. Let $X^\NA$ denote the analytification of $X$ with respect to the trivial norm. Then there is a retraction map $r_{\mcX}: X^\NA\rightarrow \Delta_{\mcX}$ such that the direct system $(r_{\mcX})_{\mcX\in \fSN}$ induces a homeomorphism $X^\NA\stackrel{\sim}{\rightarrow}  \varprojlim \Delta_{\mcX}$ (see \cite[4]{BoJ18a}). 

\begin{defn}\label{def-TC}
A normal semi-ample (resp. ample) test configuration of $(X, L)$, denoted by $(\mcX, \mcL, \eta)$ or simply by $(\mcX, \mcL)$ consists of the following data 
\begin{enumerate}
\item[(i)] $(\mcX, \mcL)$ is a model of $(X, L)$ with the $\bC^*$-action generated by a holomorphic vector field $\eta$;
\item[(ii')] The $\bC^*$-equivariant isomorphism $(\mcX, \mcL)\times_\bC\bC^*\cong (X, L)\times\bC^*$
is induced by $\eta$;
\item[(iii)]
$\mcX$ is normal and $\mcL$ is $\pi_\mcX$-semi-ample (resp. ample).
\end{enumerate}
For simplicity, we just call a normal semiample test configuration to be a test configuration.

Two test configurations $(\mcX_i, \mcL_i), i=1,2$ are called equivalent if there exists a test configuration $(\mcX_3, \mcL_3)$ and two $\bC^*$-equivariant birational morphisms $\mu_i: \mcX_3\rightarrow \mcX_i$ such that $\mu_1^*\mcL_1\cong \mu_2^*\mcL_2$.

$(X_\bC, L_\bC):=(X, L)\times\bC$ is the trivial test configuration.
($\mcX, \mcL)$ is called dominating if $\mcX\ge X_\bC$. 
\end{defn}
In this paper, for any $\bC^*$-equivariant datum $\bullet$ over $\bC$, we will use $\bar{\bullet}$ to denote its natural $\bC^*$-equivariant compactification over $\bP^1$.

Any semi-ample dominating test configuration $(\mcX, \mcL) \stackrel{p_{\mcX}}{\rightarrow} X_\bC$ defines a smooth psh non-Archimedean metric $\phi_{(\mcX, \mcL)}$ on $(X^\NA, L^\NA)$ that is represented as a function on $X^\NA$ as follows: for any $v\in X^\NA$, let $G(v)$ be the Gauss extension and set:
\begin{equation}\label{eq-phiH}
(\phi_{(\mcX, \mcL)}-\phi_\triv)(v)=G(v)(\mcL-\mu_{\mcX}^*L_{\bC}).
\end{equation}
Here $\phi_{\triv}$ is the psh non-Archimedean metric associated to the trivial test configuration $(X_\bC, L_\bC):=(X, L)\times\bC$. Note that any semi-ample test configuration is equivalent to a dominating one. Two equivalent test configurations define the same non-Archimedean metric (by definition). We always denote by $\mcH^\NA=\mcH^\NA(L)$ the set of (smooth and positive) non-Archimedean metrics coming from test configurations. By abuse of notations, we will interchangeably use the notation of test configurations and non-Archimedean metrics in $\mcH^\NA$. 
\begin{defn}[{see \cite{BFJ15, BoJ18b}}]\label{def-psh}
A psh metric on $L^\NA$ is a function $\phi: X^{\NA}\rightarrow \bR\cup \{-\infty\}$, not identically $-\infty$ that can be written as the limit of a decreasing sequence in $\mcH^\NA$. Denote by $\PSH^\NA(L)$ the set of (non-Archimedean) psh metrics on $L^\NA$. For simplicity, we also denote by $\PSH^{0,\NA}$ the space of continuous psh metrics.
\end{defn}
Non-Archimedean psh metrics were originally characterized as limits of decreasing nets of metrics from $\mcH^\NA$ (see \cite{BFJ16}). Thanks to the countable regularization result in \cite[Proposition 4.7]{BFJ15}, we can indeed use decreasing sequences as in the above definition. 

We will need the following basic property of psh metrics:
\begin{thm}[{\cite[Theorem 5.29]{BoJ18a}}]\label{thm-phidec}
For any $\phi\in \PSH^\NA(L)$, $((\phi-\phi_\triv)\circ r_\mcX)_{\mcX\in \fDS}$ is decreasing net of continuous functions, with limit $\phi-\phi_\triv$.
\end{thm}


 Through out this paper we use the following notations:
\begin{equation}
V=L^{\cdot n}, \quad K^{{\log}}_{\bar{\mcX}/\bP^1}=K_{\bar{\mcX}}+\mcX_{0,{\rm red}}-\pi^*(K_{\bP^1}+\{0\}), \quad \ud{S}=\frac{- nK_X\cdot L^{\cdot (n-1)}}{L^{\cdot n}}.
\end{equation}
For any $\phi=\phi_{(\mcX, \mcL)}\in \mcH^\NA(L)$, we follow the notations used in \cite{BHJ17} (see also \cite{Oda12}) to define:
\begin{eqnarray}
\bfE^\NA(\phi)
&:=&\frac{1}{n+1}\bar{\mcL}^{\cdot n+1}=:\bfE^\NA(\mcX, \mcL) \label{eq-ENAH}\\
\Lam^\NA(\phi)&:=&\qV \bar{\mcL} \cdot {L}_{\bP^1}^{\cdot n}
=:\Lam^\NA(\mcX, \mcL)\\
\bfJ^\NA(\phi)&:=&\Lam^\NA(\phi)-\bfE^\NA(\phi)=:\bfJ^\NA(\mcX, \mcL)\\
\bfI^\NA(\phi)&:=&\qV \bar{\mcL}\cdot L_{\bP^1}^{\cdot n}-\qV \bar{\mcL}^{\cdot n+1}+\qV \bar{\mcL}^{\cdot n}\cdot L_{\bP^1} \\
\bfR^\NA(\phi)&:=&\qV  K_{X_{\bP^1}/\bP^1}^{\log} \cdot \bar{\mcL}^{\cdot n}=:\bfR^\NA(\mcX, \mcL) \label{eq-RNA}\\
 \bfH^\NA(\phi)&:=& \qV  K^{\log}_{\bar{\mcX}/\bP^1}\cdot \bar{\mcL}^{\cdot n}-\qV  K^{\log}_{X_{\bP^1}/\bP^1}\cdot \bar{\mcL}^{\cdot n}=:\bfH^\NA(\mcX, \mcL) \label{eq-HNAH} \\
\bfS^\NA(\phi)&:=&\bfH^\NA+\bfR^\NA=\qV  K^{\log}_{\bar{\mcX}/\bP^1}\cdot \bar{\mcL}^{\cdot n}=:\bfS^\NA(\mcX, \mcL) \label{eq-SNA}\\
\cJ^\NA(\phi)&:=&(\cJ^{K_X})^\NA(\phi)=\bfR^\NA(\phi)+\ud{S}\bfE^\NA(\phi)=:\cJ^\NA(\mcX, \mcL) \label{eq-cJNA} \\
&&\nonumber\\
 \bfM^\NA(\phi)&:=& \bfH^\NA(\phi)+\bfR^\NA(\phi)+\ud{S}\bfE^\NA(\phi)\nonumber\\
 &=&\bfS^\NA(\phi)+\ud{S}\bfE^\NA(\phi)=\bfH^\NA(\phi)+\cJ^\NA(\phi)\nonumber\\
 &=&\qV K^{\log}_{\bar{\mcX}/\bP^1}\cdot \bar{\mcL}^{\cdot n}+\frac{\ud{S}}{n+1}\bar{\mcL}^{\cdot n+1}=:\bfM^\NA(\mcX, \mcL) \label{eq-MNAH}\\
\Fut(\phi)&=&\qV K_{\bar{\mcX}/\bP^1}\cdot \bar{\mcL}^{\cdot n}+\frac{\ud{S}}{n+1}\bar{\mcL}^{\cdot n+1}=:\Fut(\mcX, \mcL). \label{eq-FutNA}
\end{eqnarray}

Another important class of non-Archimedean metrics comes from bounded filtrations:
\begin{defn}\label{def-filtration}
A (graded) filtration $\mcF R_\bullet$ of the graded $\bC$-algebra $R=\bigoplus_{m=0}^{\infty}R_m=\bigoplus_{m=0}^{+\infty} H^0(X, mL)$ consists of a family of subspaces $\{\mcF^\lambda R_m\}_\lambda$ of $R_m$ for each $m\in\bZ_{\ge 0}$ satisfying:
\begin{itemize}
\item (decreasing) $\mcF^\lambda R_m\subseteq \mcF^{\lambda'}R_m$ if $\lambda\ge \lambda'$;
\item (left-continuous) $\mcF^\lambda R_m=\bigcap _{\lambda'<\lambda}\mcF^{\lambda'} R_m$;
\item (multiplicative) $\mcF^\lambda R_m\cdot \mcF^{\lambda'}R_{m'}\subseteq \mcF^{\lambda+\lambda'}R_{m+m'}$, for any $\lambda, \lambda'\in \bR$ and $m, m'\in \bZ_{\ge 0}$;
\item (linearly bounded) There exists $e_-, e_+\in \bZ$ such that $\mcF^{me_-}R_m=R_m$ and $\mcF^{me_+}R_m=0$ for all $m\in \bZ_{\ge 0}$.
\end{itemize}
\end{defn}
According to \cite[3.4]{BoJ18b}, graded filtrations are in bijection with bounded graded norms on $R$. 
Given such a filtration, for any $m\in \bZ_{\ge 0}$, $\mcF R_m$ generates a finitely generated filtrations that determines a metric $\check{\phi}_{\cF, m}=m^{-1}\FS(\mcF R_m)$. They form an increasing sequence converging to a bounded metric $\phi_{\cF}=\FS(\cF R_\bullet)$. For more details, see \cite{BoJ18b, Fuj18, Sze15}. 


\begin{thm}[\cite{BoJ18b}]\label{thm-lowerreg}
\begin{enumerate}
\item
A psh metric $\phi\in \PSH^\NA(L)$ is of the form $\phi_\cF$ for a filtration $\cF$ if and only if it is lower regularizable, which means that there exists an increasing net $\{\phi_m\}\subset \mcH^{\NA}(L)$ that converges to $\phi$.
\item
A continuous metric is always lower regularizable.
\end{enumerate}
\end{thm}
By abuse of notations, we will also interchangeably use the notation $\cF$ and its associated non-Archimedean metric. 
For simplicity we will denote by $\PSH^{\cF,\NA}$ to denote the space of non-Archimedean psh metrics that are associated to filtrations, i.e. the space of lower regularizable non-Archimedean psh metrics.
We will obtain existence results for cscK metrics by using a special class of continuous metrics associated to the following class of filtrations. 
\begin{defn}\label{def-modfil}
For any normal model $(\mcX, \mcL)$ of $(X, L)$, we define its associated filtration $\mcF_{(\mcX, \mcL)}=\{\mcF^{\lambda}_{(\mcX, \mcL)}R_m\}$ as:
\begin{equation}
\mcF_{(\mcX, \mcL)}^\lambda R_m=\left\{s\in H^0(X, mL); t^{-\lceil \lambda \rceil} \bar{s} \in H^0(\mcX, \lfloor m\mcL\rfloor)\right\}.
\end{equation}
Filtrations obtained in this way will be called model filtrations. We will denote by $\phi_{(\mcX, \mcL)}$, or simply by $\phi_{\mcL}$, the non-Archimedean psh metrics associated to $\mcF_{(\mcX, \mcL)}$, and by $\PSH^{\fM, \NA}$ the space of non-Archimedean metrics associated to model filtrations.
\end{defn}
Note that here we do not require $\mcL$ to be semiample. Model filtrations can be described in a different way. First we can assume that $\pi: (\mcX, \mcL)\rightarrow \bC$ is a dominating model with a $\bC^*$-equivariant birational morphism $\rho: \mcX\rightarrow X_\bC$. Write $\mcL=\rho^*L_\bC+D$ and set $\mcI'_m=\rho_*(\mcO_{\mcX}(mD))$ 
$\mcI'_m$ is an integrally closed fractional ideal of $X_{\bC}$ which has the shape:
\begin{equation}
\mcI'_m=\sum_{\lambda\in \bZ} t^{-\lambda} I'_{m,\lambda}.
\end{equation}
Then we have the identity (see \cite[section 2.6]{BHJ17}):
\begin{equation}
\mcF^\lambda R_m=H^0\left(X, I'_{m, \lceil \lambda\rceil} \otimes_{\mcO_X}  mL \right).
\end{equation}
Note that $I_{m,\lceil \lambda\rceil}\otimes mL$ is however in general not globally generated since $\mcL$ is not assumed to be semiample.

If $f_\mcL$ denotes the model function that is defined by: for any $v\in X^{\rm div}_\bQ$,
\begin{equation}
f_\mcL(v)=G(v)(D).
\end{equation}
The $\phi_\triv$-psh upper envelope of $f_\mcL$ is defined as:
\begin{equation}\label{eq-upenv}
P(f_\mcL)(v)=\sup\left\{(\phi-\phi_{\triv})(v); \phi\in \PSH^\NA(L), \phi-\phi_\triv\le f_\mcL \right\}.
\end{equation}
By \cite[Theorem 8.5]{BFJ16} we have the equality $\phi_{(\mcX, \mcL)}=\phi_\triv+P(f_\mcL)$, which is always continuous by \cite[Theorem 8.3]{BFJ16}. More concretely, if $\fb_m$ denotes the $\pi$-relative base ideal of $m\mcL$ and
$\mu_m: \mcX_m\rightarrow \mcX$ is the normalized blowup of $\fb_m$ with the exceptional divisor denoted by $E_m$, then 
$(\mcX_m, \mcL_m=\mu_m^*\mcL-\frac{1}{m}E_m)$ is a semiample test configuration and we have
\begin{equation*}
\phi_{(\mcX, \mcL)}=\lim_{m\rightarrow+\infty} \phi_{(\mcX_m, \mcL_m)}.
\end{equation*} 
Moreover if we write
\begin{equation}
\mcI_m=(\rho\circ \mu_m)_*\mcO_{\mcX_m}(mD-E_m)=\sum_{\lambda\in \bZ}t^{-\lambda} I_{m,\lambda},
\end{equation}
with $I_{m,\lambda}\supseteq I_{m, \lambda'}$ if $\lambda'\ge \lambda$,
then $I_{m,\lambda}\otimes mL$ is globally generated and we also have:
\begin{equation}
\mcF^\lambda R_m=H^0(X, I_{m,\lceil\lambda\rceil}\otimes mL).
\end{equation}

The most well-studied model filtrations are those from test configurations.
\begin{exmp}[\cite{BHJ17, Sze15, WN12}]
For any dominating test configuration $(\mcX, \mcL)$ via dominating morphism $\rho: \mcX\rightarrow X_\bC$. Assume $\mcX_0=\sum_{i=1}^I b_i E_i$ and $\mcL=\rho^*L_\bC+D$. There is an associated model filtration (see \cite[Lemma 5.17]{BHJ17}):
\begin{eqnarray*}
\mcF^\lambda R_m&=&\left\{s\in H^0(X, mL); t^{-\lceil \lambda\rceil} \bar{s} \text{ extends to a holomorphic section of } m\mcL\right\}\\
&=&\left\{s\in H^0(X, mL);  r(\ord_{E_i})(s)+m b_i \cdot \ord_{E_i}D\ge b_i \lceil\lambda\rceil, i=1,\dots, I \right\}.
\end{eqnarray*}
\end{exmp}

\subsubsection{Finite energy metrics and measures}
Boucksom-Jonsson developed a non-Archimedean approach to K-stability, which is a natural set-up for studying and compactifying the space of (equivalent classes of) test configurations. In particular the space of non-Archimedean metrics with finite energy space, which is a natural compactification of $\mcH^\NA$, was introduced in \cite{BoJ18a}:
\begin{defn}\label{def-cE1NA}
For any $\phi\in {\rm PSH}^\NA(L)$, define
\begin{equation}
\bfE^\NA(\phi)=\inf\{\bfE^\NA(\tilde{\phi}); \tilde{\phi}\in \mcH^\NA(L) \text{ and } \tilde{\phi}\ge \phi\}.
\end{equation}
Set
\begin{equation}
\cE^{1,\NA}:=\cE^{1,\NA}(L)=\{\phi\in {\rm PSH}^\NA(L); \bfE^\NA(\phi)> -\infty\}.
\end{equation}
A sequence $\{\phi_m\}$ in $\mcE^{1,\NA}(L)$ converges to $\phi\in \mcE^{1,\NA}(L)$ in the strong topology if $\lim_{m\rightarrow+\infty} (\phi_m-\phi)=0$ on $X^{\rm qm}$ (set of quasi-monomial points in $X^\NA$) and $\lim_{m\rightarrow+\infty} \bfE^\NA(\phi_m)=\bfE^\NA(\phi)$. The strong topology on $\cE^{1,\NA}/\bR$ is by definition the quotient topology induced by the strong topology on $\cE^{1,\NA}(L)$.
\end{defn}

By the work of Boucksom-Jonsson \cite{BoJ18a}, we have the mixed Monge-Amp\`{e}re energy functional for finite energy non-Archimedean metrics. 
\begin{thm}[{\cite[Theorem 6.9]{BoJ18a}}]\label{thm-BoJ}
There exists a unique operator:
\begin{equation}
(\phi_1, \dots, \phi_n)\mapsto \MA^\NA(\phi_1, \dots, \phi_n).
\end{equation}
taking an $n$-tuple in $\mcE^{1,\NA}(L)$ to a Radon measure on $X^\NA$ such that
\begin{enumerate}
\item[(i)] If $\phi_i=\phi_{(\mcX, \mcL_i)}\in \mcH^\NA(L)$ where $(\mcX, \mcL_i)$ is a test configuration of $(X, L)$ and $\mcX_0=\sum_j b_j E_j$ is Cartier, then 
\begin{equation}\label{eq-MAmeas}
\MA^\NA(\phi_1,\dots, \phi_n)=\sum_j b_j \left(\mcL_1|_{E_j}\cdot \dots \cdot \mcL_n|_{E_j}\right)\delta_{x_j},
\end{equation}
where $x_j=b_j^{-1} r(\ord_{E_j})$ (called the Shilov point associated to $E_i$ in \cite[section 1.4]{BoJ18a}) and $r: (X\times\bC)^{\rm div}_\bQ \rightarrow X^{\rm div}_\bQ$ is the restriction map.
\item[(ii)] $\int_{X^\NA} (\phi-\phi_{\triv})\MA^\NA(\phi_1,\dots, \phi_n)>-\infty$ when $\phi, \phi_1,\dots, \phi_n\in \cE^{1,\NA}(L)$.
\item[(iii)] For any decreasing nets $\phi^j\rightarrow \phi$ and $\phi^j_i\rightarrow \phi_i$ in $\mcE^{1,\NA}(L)$, we have the convergence:
\begin{equation*}
\int_{X^\NA}(\phi^j-\phi_{\triv})\MA^\NA(\phi^j_1,\dots, \phi^j_n)\longrightarrow \int_{X^\NA}(\phi-\phi_{\triv})\MA^\NA(\phi_1,\dots, \phi_n).
\end{equation*}
\end{enumerate}
\end{thm}
We will write $\MA^\NA(\phi)$ for $\MA^\NA(\phi, \dots, \phi)$, and $\MA^\NA(\phi_1^{[k_1]},\dots, \phi_p^{[k_p]})$ for $$\MA^\NA(\overbrace{\phi_1, \dots, \phi_1}^{k_1},\dots, \overbrace{\phi_p, \dots, \phi_p}^{k_p}).$$

We recall the resolution of non-Archimedean Monge-Amp\`{e}re equations by Boucksom-Favre-Jonsson \cite{BFJ16} and extended by Boucksom-Jonsson \cite{BoJ18a}.
\begin{defn}[{\cite[7.1, 7.5]{BoJ18a}}]\label{def-M1NA}
The energy of a positive radon measure (see \cite[Chapter 7]{Fol99}) $\nu$ on $X^\NA$ (with mass $V$) is 
\begin{equation}\label{eq-E*NA}
\bfE^{*\NA}(\nu)=\sup_{\phi\in \cE^{1,\NA}(L)}\left(\bfE^\NA(\phi)-\qV \int_{X^\NA}(\phi-\phi_{\triv})\nu \right)\in \bR\cup \{+\infty\}.
\end{equation}
We say that $\nu$ has finite energy if $\bfE^{*\NA}(\nu)<+\infty$, and denote by $\mcM^{1,\NA}$ the set of (positive) radon measure (with mass $V$) of finite energy on $X^\NA$.

A net $\{\nu_j\}_j$ in $\mcM^{1,\NA}$ converges strongly to $\nu$ iff $\nu_j\rightarrow \nu$ weakly and $\lim_{j\rightarrow+\infty}\bfE^{*\NA}(\nu_j)=\bfE^{*\NA}(\nu)$.
\end{defn}

\begin{thm}[{non-Archimedean Calabi-Yau theorem, \cite[Theorem A]{BFJ16}, \cite[Theorem 7.3, 7.25]{BoJ18a}}]\label{thm-NACY}
The Monge-Amp\`{e}re operator defines a homeomorphism 
\begin{equation}
\MA^\NA: \cE^{1,\NA}(L)/\bR\rightarrow \mcM^{1,\NA}
\end{equation}
with respect the strong topology. Moreover, if $\nu$ is a Radon measure with mass $V$ supported on a dual complex $\Delta_\mcX$ for a SNC model $\mcX$, then $(\MA^\NA)^{-1}(\nu)$ is continuous.
\end{thm}

All the Archimedean energy functionals in \eqref{eq-EE1}-\eqref{eq-defI} can be defined by replacing the Archimedean integrals in section \ref{sec-Archrays} by the corresponding non-Archimedean integrals. 
For example we have
\begin{defn}
For any $\phi\in \cE^{1,\NA}(L)$, define 

\begin{equation}
\Lam^\NA(\phi)=\qV \int_{X^\NA} (\phi-\phi_\triv)\MA^\NA(\phi_\triv).
\end{equation}

\end{defn}

We will also consider a more general mixed Monge-Amp\`{e}re measure.
\begin{defn}
 Let $\pi_{\mcY}: \mcY\rightarrow \bC$ be a model of $X$ (see Definition \ref{def-TC}) and $\mcQ$ be a $\bC^*$-equivariant $\bQ$-line bundle over $\mcY$. We think of $\mcQ$ as a non-Archimedean metric $\phi_{\mcQ}$ on $(X^\NA, Q^\NA)$ where $Q=\cQ|_{X\times\{1\}}$. 
First assume that $\mcQ$ is $\pi_\mcY$-semiample, define for any $\phi_i\in \cE^{1,\NA}, 1\le i\le n$:
\begin{eqnarray}\label{eq-mixed}
\MA^\NA(\phi_\mcQ, \phi_1, \cdots, \phi_{n-1})&:=&\frac{1}{n!}\frac{\partial^{n-1}}{\partial t_1\cdots \partial t_{n-1}}\MA^\NA (\phi_{\mcQ}+\sum_{i=1}^{n-1} t_i \phi_i).
\end{eqnarray}
where we wrote $\phi_0=\phi_{\mcQ}$. 
In general, we write $\mcQ=\mcQ_1-\mcQ_2$ with $\mcQ_i$ being $\pi_\mcY$-semiample and define:
\begin{equation}
\MA^\NA(\phi_\mcQ, \phi_1, \cdots, \phi_{n-1}):=\MA^\NA(\phi_{\mcQ_1}, \phi_1,\dots, \phi_{n-1})-\MA^\NA(\phi_{\mcQ_2}, \phi_1,\dots, \phi_{n-1}).
\end{equation}
\end{defn}
Note that in \eqref{eq-mixed} we used the fact that the mixed non-Archimedean Monge-Amp\`{e}re operator is linear and symmetric with respect to the variables. Indeed, if $\phi_i\in \mcH^\NA$, the mixed Monge-Amp\`{e}re operator on the left-hand-side of \eqref{eq-mixed} can be defined using an intersection formula similar to \eqref{eq-MAmeas} and is easily shown to be equal to the right-hand-side. For general $\phi_i\in \cE^{1,\NA}$, we can choose decreasing sequences $\{\phi_{i,k}\}_k\subset \mcH^\NA$ converging to $\phi_i$ and show that both sides converge to the same limit as $k\rightarrow+\infty$.
In fact it is well known that one can also use a polarization formula to define the mixed Monge-Amp\`{e}re measure (setting $\phi_0=\phi_{\mcQ}$):
\begin{equation*}
\MA^\NA(\phi_\mcQ, \phi_1, \cdots, \phi_{n-1})=\frac{1}{n!}\sum_{k=1}^n\sum_{0\le i_1<\cdots<i_{k}\le n-1}(-1)^{n-k}\MA^\NA(\phi_{i_1}+\cdots+\phi_{i_k}). 
\end{equation*}

\begin{defn}\label{def-EchiNA}
Let $(\mcY, \mcQ)$ be the data as in the above definition.
For any $\phi\in \cE^{1,\NA}(L)$, define:
\begin{equation}\label{eq-EcQNA}
(\bfE^{\cQ})^\NA(\phi):=\qV \sum_{k=0}^{n-1}\int_{X^\NA}(\phi-\phi_{\triv})\MA^\NA(\phi_\cQ, \phi^{[k]}, \phi_{\triv}^{[n-1-k]}).
\end{equation}
In particular, when $(\mcY, \mcQ)=(X_\bC:=X\times\bC, K_{X_\bC/\bC}^{\log}=p_1^*K_X)$ we define:
\begin{equation}\label{eq-RNA2}
\bfR^\NA(\phi):=\bfE^{K^{\log}_{X_\bC/\bC}}(\phi).
\end{equation}
\end{defn}
The following lemma, which says that non-Archimedean energy functionals reduce to the corresponding intersection products, can be verified directly using Theorem \ref{thm-BoJ}.
\begin{lem}
With the notations in the above definition, if $\phi=\phi_{(\mcX, \mcL)}$ for $(\mcX, \mcL)\in \mcH^\NA(X, L)$ and $\mcX$, then 
\begin{eqnarray}
\bfE^\NA(\phi)&=&\qV \frac{1}{n+1} \sum_{k=0}^n\int_{X^\NA}(\phi-\phi_{\triv})\MA(\phi^{[k]},\phi_\triv^{n-k})=\qV \frac{\bar{\mcL}^{\cdot n+1}}{n+1}\\
 (\bfE^\cQ)^\NA(\phi)&=&\qV  \bar{\cQ}\cdot \left(\bar{\mcL}^{\cdot n}-L_{\bP^1}^{\cdot n}\right), \label{eq-intEcQNA}
\end{eqnarray}
where the intersection in \eqref{eq-EcQNA} is calculated on a common refinement of $\mcX$ and $\mcY$.
\end{lem}
Because of identity \eqref{eq-intEcQNA}, for any $\phi\in \mcE^{1,\NA}(L)$, we will simply write:
\begin{equation}\label{eq-Qdotphi1}
(\bfE^\mcQ)^\NA(\phi)=\bar{\mcQ}\cdot (\phi^{\cdot n}-\phi_{\triv}^{\cdot n}).
\end{equation}

Note that if $\mcQ=\pi^*Q$ for a $\bQ$-line bundle over $X$, then because $Q_{\bP^1}\cdot \phi_\triv^{\cdot n}=0$, we have:
\begin{equation}\label{eq-Qdotphi0}
(\bfE^Q)^\NA(\phi):=(\bfE^{\pi^*Q})^\NA(\phi)=Q_{\bP^1}\cdot \phi^{\cdot n}.
\end{equation}
Using the property of mixed Monge-Amp\`{e}re operators and the non-Archimedean estimates developed in \cite{BoJ18a}, it is easy to adapt the proof in the Archimedean case (see \cite{BBGZ13, DR17}) to prove the following useful result.
\begin{prop}(see \cite{BoJ18a})\label{prop-strongcont}
The energy functionals $\Lam^\NA$ and $(\bfE^\cQ)^\NA$ map $\mcE^{1,\NA}(L)$ to $\bR$ and are continuous with respect to the convergence in strong topology. 
\end{prop}

Recall that there is a log discrepancy function $A_X: X^{\NA}\rightarrow \bR_{\ge 0}\cup \{+\infty\}$ that extends the usual log discrepancy function for divisorial valuations. By \cite[Theorem 2.1]{BoJ18a} we have the identity:
\begin{equation}\label{eq-AXsup}
A_X=\sup_{\mcY\in \fSN}A_X\circ r_{\mcY}.
\end{equation}
Boucksom-Jonsson defined the following non-Archimedean entropy functional (\cite[2.4]{BoJ18b}): for any $\phi\in \mcE^{1,\NA}(L)$ set :
\begin{equation}\label{eq-HNAphi1}
\bfH^\NA(\phi)=\qV \int_{X^\NA}A_X(x)\MA^\NA(\phi)(x).
\end{equation}
So for any $\phi\in \mcE^{1,\NA}(L)$, we can also define:
\begin{equation}\label{eq-MNAE1}
\bfM^\NA=\bfH^\NA+\bfR^\NA+\frac{\ud{S}}{n+1}\bfE^\NA.
\end{equation}

So up to now all the non-Achimedean functionals in \eqref{eq-ENAH}-\eqref{eq-MNAH} have been defined for all $\phi\in \cE^{1,\NA}(L)$. However the key issue in the study of YTD conjecture is that the functional $\bfH^\NA$ is in general not continuous with respect to the convergence in the strong topology (see \cite{BoJ18b}).

\subsubsection{$\bG$-uniform K-stability}\label{sec-twist}

Let $\bG$ be a reductive complex Lie group in $\Aut(X, L)_0$ with a maximal compact subgroup $\bK$ such that $\bK^\bC=\bG$.
Let $\bT=((S^1)^r)^\bC$ be the identity component of the center of $\bG$. We set:
\begin{eqnarray*}
(\mcH^\NA)^\bK&:=&\{\phi_{(\mcX, \mcL)}; (\mcX, \mcL) \text{ is a $\bG$-equivariant  test configuration }\}\\
(\mcE^{1,\NA})^\bK&:=&\left\{\phi\in \cE^{1,\NA}; \phi=\lim_{m\rightarrow+\infty}\phi_m  \text{ for a decreasing sequence } \{\phi_m\}\in (\mcH^\NA)^\bK  \right\}
\end{eqnarray*}
Set $N_\bZ={\rm Hom}(\bC^*, \bT)$, $N_\bR=N_\bZ\otimes_\bZ \bR$, $M_\bZ={\rm Hom}(N_\bZ, \bZ)$ and $M_\bR=M_\bZ\otimes_\bZ \bR$. $N_\bR$ has a natural action on $(X^\NA)^\bT$. 
This can be studied using Berkovich's notion of peak points (see \cite[section 4.2, 4.3]{BPS14} for the toric case and the corresponding study of toric metrics). For our purpose, we just need the following description of the $N_\bR$-action on $X^{\rm div}_\bQ$.
By the structure theory of $\bT$-varieties (see \cite{AHS08}), $X$ is birationally a torus fibration over the Chow quotient of $X$ by $\bT$ which will be denoted by $X/\!/\bT$. 
As a consequence the function field $\bC(X)$ is the quotient field of the Laurent polynomial algebra:
\begin{equation}
\bC(X\sslash \bT)[M_\bZ]=\bigoplus_{\alpha\in M_\bZ}\bC(X\sslash \bT)\cdot 1^{\alpha}.
\end{equation}
Given a valuation $\nu$ of the functional field $\bC(X\sslash \bT)$ and a vector $\lambda\in N_\bR$, we obtain a valuation  (\cite[page 236]{AHS08}): 
\begin{equation}
v_{\nu, \lambda}: \bC[X\sslash \bT][M_\bZ]\rightarrow \bR, \quad \sum_{i} f_i\cdot 1^{\alpha_i}\quad \mapsto\quad \min\left(\nu(f_i)+\la \alpha_i, \lambda \ra \right).
\end{equation}
The vector space $N_\bR$ acts on $X^{\rm div}_\bQ$ in the following natural way. If $v=\nu_{\nu,\lambda}$, then
\begin{equation}
\xi\circ v= \xi \circ v_{\nu, \lambda}= v_{\nu, \lambda+\xi}=:v_\xi.
\end{equation}
On the other hand, we have an action of $N_\bR$ on the space of test configurations:
\begin{defn}[\cite{His16b}]
Let $(\mcX, \mcL, \eta)$ be a $\bG$-equivariant test configuration. For any $\xi\in N_\bR$, the $\xi$-twist of $(\mcX, \mcL, \eta)$, denoted also by $(\mcX, \mcL)_\xi$, is the data $(\mcX, \mcL, \eta+\xi)$. 
\end{defn}
\begin{lem}
If $(\mcX, \mcL)$ is a $\bG$-equivariant test configuration, then for any $\xi\in N_\bZ$, $(\mcX, \mcL)_\xi$ is a test configuration and $\phi_\xi:=\phi_{(\mcX, \mcL)_\xi}$ satisfies the identity:
\begin{equation}\label{eq-phixi}
(\phi_\xi-\phi_\triv)(v)=(\phi-\phi_\triv)(v_\xi)+\theta_\xi(v)
\end{equation}
where $\theta_\xi$ is equal to $\phi_{\triv, \xi}-\phi_\triv$. Moreover for any $\phi\in (\mcH^\NA)^\bK, 1\le i\le n$, we have the identity:
\begin{equation}\label{eq-ximeas}
\MA^\NA(\phi_{1,\xi},\dots, \phi_{n,\xi})=(-\xi)_* \MA^\NA(\phi_1,\dots, \phi_n).
\end{equation}
\end{lem}

\begin{proof}
We follow the similar proof as in \cite[Proof of Proposition 3.3]{Li19}. Consider the commutative diagram where $\bar{\sigma}_\xi$ is the $\bC^*$-action generated by $-\xi$ on $X\times\bC^*$.
\begin{equation}\label{eq-3morph}
\xymatrix{
& \ar_{q_1}[ld] \mcU \ar^>>>>>>>{\pi_\mcW}[dd]  \ar^{q_2}[rd] & \\
\mcX=\mcX^{(1)} \ar_{\pi_1}[dd] \ar@{-->}[rr] &   & \mcX=\mcX^{(2)} \ar^{\pi_2}[dd]   \\
 & \ar_{p_1}[ld] \mcW \ar^{p_2}[rd] & \\
X_\bC=X^{(1)}_\bC \ar^{\bar{\sigma}_\xi}@{-->}[rr]  &  & X_\bC=X^{(2)}_\bC 
}
\end{equation}
The map $\pi_1\circ q_1$ is $\eta$-equivariant. Moreover, the test configuration $(\mcX, \mcL)_\xi$ is equivalent to the test configuration $(\mcU, q_2^*\mcL, \eta)$.
We now decompose:
\begin{eqnarray}\label{eq-twistNA}
q_2^*\mcL-q_1^*\pi_1^*L_\bC&=&q_2^*\mcL-q_2^*\pi_2^*L_\bC+q_2^*\pi_2^*L_\bC-q_1^*\pi_1^*L_\bC\nonumber\\
&=& q_2^*(\mcL-\pi_2^*L_\bC)+\pi_{\mcW}^*(p_2^*L_\bC-p_1^*L_\bC). \label{eq-twistdec}
\end{eqnarray}
For any $w\in X^{\rm div}_\bQ$, for any $f\in \bC(X)_\alpha$, let $\bar{f}=p_1^*f$ denote the function on $X\times \bC^*$ via the projection $p_1$ to the first factor. Then $\bar{\sigma}_\xi^* \bar{f}=t^{\la \alpha, \xi\ra} \bar{f}$. By the definition of Gauss extension, we get:
\begin{eqnarray*}
(q_2)_*G(w)(\bar{f})&=&G(w)((q_2)^*\bar{f})=G(w)(t^{\la \alpha, \xi\ra} \bar{f})=\la \alpha, \xi\ra+G(w)(\bar{f})\\
&=&G(w_\xi)(\bar{f}).
\end{eqnarray*}
So $(q_2)_*G(w)=G(w_\xi)$. This together with with decomposition \eqref{eq-twistNA} gives the identity \eqref{eq-phixi}.

Let $v_j\in X^\NA$ be the point satisfying $b_j G(v_j)=\ord_{E_j}$. Then by the same calculation we have for $w\in X^{\rm div}_\bQ$,
\begin{eqnarray*}
G(w)(q_1^*\bar{f})=G(w_{-\xi})(\bar{f}),
\end{eqnarray*}
which means that the point associated to $E_j$ in the twisted test configuration is given by $v_{j,-\xi}$. 
So we use the formula \eqref{eq-MAmeas} for Monge-Amp\`{e}re measure to get:
\begin{equation}
\MA(\phi_\xi)=\sum_i b_i (\mcL|_{E_i})^n \delta_{v_{j,-\xi}}=(-\xi)_* \MA(\phi).
\end{equation}
Clearly, the case of mixed Monge-Amp\`{e}re measure can be proved in a similar way.
\end{proof}
\begin{defn}
For any $\phi\in \PSH^\NA(L)$,  $\phi_\xi\in \PSH^\NA(L)$ is defined by the formula \eqref{eq-phixi}. 
\end{defn}
\begin{exmp}\label{exm-cFtwist}
If $\cF$ is a $\bT$-equivariant filtration, then we have a weight decomposition 
\begin{equation}
\cF^x R_m=\bigoplus_{\alpha\in M_\bZ}(\cF^x R_m)_\alpha.
\end{equation}
The $\xi$-twist of $\cF$, denoted by $\cF_\xi$, is defined by the decomposition:
\begin{equation}
\cF^x_\xi R=\bigoplus_{\alpha\in M_\bZ} (\cF^x_\xi R_m)_\alpha, \quad (\cF_\xi^x R_m)_\alpha:=(\cF^{x-\la\alpha,\xi\ra}R_m)_\alpha.
\end{equation}
By \cite[3.2]{Li19}, we have the identity $\phi_{\cF_\xi}=(\phi_\cF)_\xi$.
\end{exmp}

By approximating $\phi$ by a decreasing sequence from $\mcH^\NA$, one can check that $\phi_\xi$ is indeed a well-defined metric. 
Moreover with the approximation argument, we also get:
\begin{cor}\label{cor-ximeas}
For any $\phi_i\in (\cE^{1,\NA})^\bK, 1\le i\le n$, we have 
\begin{equation}
\MA(\phi_{1,\xi}, \dots, \phi_{n,\xi})=(-\xi)_* \MA(\phi_1,\dots, \phi_n).
\end{equation}
\end{cor}
Assume $R_m=H^0(X, mL)=\bigoplus_j (R_m)_{\alpha^{(m)}_j}$ is a decomposition into weight spaces with respect to the $\bT$-action.
For any $\xi\in N_\bR$, we define the Chow weight of $\xi$ as
\begin{equation}
\chw_L(\xi)=\lim_{m\rightarrow+\infty}\frac{1}{m^n/n!}\sum_{j}\frac{\la \alpha^{(m)}_j, \xi\ra}{m} \dim (R_m)_{\alpha^{(m)}_j}.
\end{equation}

The following lemma gives transformation formula of non-Archimedean functionals under the twists by elements from $N_\bR$.

\begin{lem}
For any $\phi\in \cE^{1,\NA}$ and $\xi\in N_\bR$, we have the following identities: 
\begin{eqnarray*}
\bfE^\NA(\phi_\xi)&=&\bfE^\NA(\phi)+\chw_L(\xi).
\end{eqnarray*}
Moreover for any $\phi_i\in \cE^{1,\NA}, i=1,2$, we have the identity:
\begin{equation}\label{eq-Ixisame}
\bfI(\phi_{1,\xi}, \phi_{2,\xi})=\bfI(\phi_1,\phi_2). 
\end{equation} 
\end{lem}
\begin{proof}
Using approximations by smooth decreasing sequences, we can assume $\phi\in \mcH^\NA$. Moreover, we can assume $\xi\in N_\bZ$ by approximation and base change (see \cite{Li19}). 
The first identity is proved in the same way as in \cite{Li19}. Next we use the identity \eqref{eq-phixi} and Corollary \ref{cor-ximeas}  to get:
\begin{eqnarray*}
\int_{X^\NA}(\phi_{2,\xi}-\phi_{1,\xi})\MA(\phi_{1,\xi}^{[k]}, \phi_{2,\xi}^{[n-k]})&=&\int_{X^\NA}\left((\phi_2-\phi_1)\circ \xi\right) \left[(-\xi)_*\MA(\phi_1^{[k]}, \phi_2^{[n-k]})\right]\\
&=&\int_{X^\NA}(\phi_2-\phi_1)\MA^\NA(\phi_1^{[k]}, \phi_2^{[n-k]}).
\end{eqnarray*}
By the formula for $\bfI$, this easily implies \eqref{eq-Ixisame}.
\end{proof}
\begin{lem}
For any $\phi\in (\mcH^\NA)^\bK$ and $\xi\in N_\bR$ we have the identity:
\begin{equation}
\bfM^\NA(\phi_\xi)=\bfM^\NA(\phi)+\Fut(\xi).
\end{equation}
\end{lem}
We believe that this to be true for any $\phi\in \cE^{1,\NA}$ and this would follow from Conjecture \ref{conj-reg}.
\begin{proof}
For any $\bC^*\times\bT$-equivariant SNC model $\mcY$ of $X$ that equivariantly dominates $\mcX$, we set
\begin{equation}
f_{\mcY}(\xi):=K^{\log}_{\overline{\mcY_\xi}/\bP^1} \cdot (\bar{\mcL}_{\xi})^{\cdot n}+ \frac{\ud{S}}{n+1} (\bar{\mcL}_{\xi})^{\cdot n}.
\end{equation}
We claim that 
\begin{equation}\label{eq-difFut}
f_\mcY(\xi)-f_\mcY(0)=\Fut(\xi).
\end{equation} 
By projection formula for intersection numbers, we can assume $\mcX=\mcY$ and $\mcL$ is a $\bC^*\times \bT$-equivariant semiample line bundle. Then we use the commutative diagram \eqref{eq-3morph}. As in \cite[Proof of Proposition 3.3]{Li19}, we write:
$\mcL=\pi^*L+E$ and set $\mcL_b=\pi^*L+bE$.

Consider 
\begin{equation}
h(b):=(K^{\log}_{\bar{\mcU}/\bP^1}\cdot \overline{q_2^*\mcL_b}^{\cdot n}+\frac{\ud{S}}{n+1}\overline{q_2^*\mcL_b}^{\cdot n+1})
-(K^{\log}_{\bar{\mcU}/\bP^1}\cdot \overline{q_1^*\mcL_b}^{\cdot n}+\frac{\ud{S}}{n+1}\overline{q_1^*\mcL_b}^{\cdot n+1}), 
\end{equation}
where the compactifications we use are using the isomorphism induced by $\eta$. We calculate:
\begin{eqnarray*}
\frac{b}{db}h(b)&=&K^{\log}_{\bar{\mcU}/\bP^1}\cdot n q_2^*\mcL_b^{\cdot n-1}\cdot q_2^*E+\frac{\ud{S}}{n+1}q_2^*\mcL_b^{\cdot n}\cdot q_2^*E \\
&&-(K^{\log}_{\bar{\mcU}/\bP^1}\cdot n q_2^*\mcL_b^{\cdot n-1}\cdot q_2^*E+\frac{\ud{S}}{n+1}q_2^*\mcL_b^{\cdot n}\cdot q_2^*E) \\
&=&0.
\end{eqnarray*}
So we get $h(b)=h(0)=f_{L_\bC}(\xi)$. On the other hand, it is easy to see that when $(\mcX, \mcL)=(X_\bC, L_\bC)$, the test configuration $(X_\bC, L_\bC)_\xi$ is equivalent to the product test configuration 
induced by the holomorphic vector field corresponding to $\xi$. So we get in this case $h(0)=\Fut(\xi)=h(1)$ which verifies \eqref{eq-difFut}.
\end{proof}


For any $\phi\in \cE^{1,\NA}$, we set:
\begin{equation}
\bfJ^\NA_\bT(\phi)=\inf_{\xi\in N_\bR} \bfJ^\NA(\phi_\xi).
\end{equation}
\begin{defn}
\label{def-Kst}
Let $\bG$ be a connected reductive subgroup of $\Aut(X, L)_0$. 
$(X, L)$ is called $\bG$-uniformly K-stable over $\cE^{1,\NA}$ (resp. over $\PSH^{\fM,\NA}$ (see Definition \ref{def-modfil}), resp. over $\mcH^{\NA}$) if there exists $\gamma>0$ such that for any $\phi\in (\mcE^{1,\NA})^\bK$ (resp. $\phi\in (\PSH^{\fM,\NA})^\bK=\PSH^{\fM,\NA}\cap (\cE^{1,\NA})^\bK$, resp. $\phi\in (\mcH^\NA)^{\bK}$)
\begin{equation}
\bfM^\NA(\phi)\ge \gamma\cdot \bfJ^\NA_\bT(\phi)=\gamma \cdot \inf_{\xi\in N_\bR} \bfJ^\NA(\phi_\xi).
\end{equation}
To conform with the usual notations in the literature, we will also call ``$\bG$-uniformly K-stable over $\mcH^\NA$" simply as ``$\bG$-uniformly K-stable", and  ``$\bG$-uniformly K-stable over $\PSH^{\fM,\NA}$" as being ``$\bG$-uniformly K-stable for model filtrations". Moreover, if $\bfG=\{e\}$, we just call $\bG$-uniform stability to be uniform stability.
\end{defn}

It follows from the definitions that we have:
\begin{eqnarray*}
\text{$\bG$-uniformly K-stable over } \cE^{1,\NA} &\Longrightarrow&  \text{$\bG$-uniformly K-stable over } \PSH^{\fM, \NA}\\
& \Longrightarrow& \text{$\bG$-uniformly K-stable over } \mcH^\NA. 
\end{eqnarray*}
The above stability notions are conjectured to be equivalent to each other when the reductive group $\bG$ contains a maximal torus of $\Aut(X, L)_0$. $\bG$-uniform K-stability for filtrations can be seen a modified version of the stability condition for filtrations introduced by Sz\'{e}k{e}lyhidi (\cite{Sze15}).
On the other hand, the usual $\bG$-uniform K-stability was essentially introduced by Hisamoto in \cite{His16b} based Darvas-Rubinstein's principle for proving Tian's properness conjecture (\cite{DR17}).  It refines the (uniform) K-stability (see \cite{Tia97, Don02, BHJ17, Der16}). 
According to \cite[Lemma 2.2]{Tia14} and \cite{BHJ19}, by using base change to make the central fibre reduced, this definition is equivalent to the definition of $\bG$-uniform K-stability via the Futaki invariants for test configurations.
We state the $\bG$-uniform version of Yau-Tian-Donaldson conjecture.
\begin{conj}[YTD Conjecture]\label{conj-YTD}
A polarized manifold $(X, L)$ admits a cscK metric if and only if it is $\Aut(X,L)_0$-uniformly K-stable.
\end{conj}
Based on Darvas-Rubinstein's principle (\cite{DR17}) and his slope formula for $\bfJ_\bT(\vphi)=\inf_{\sigma\in \bT}\bfJ(\sigma^*\vphi)$ (see \eqref{eq-defJ}), Hisamoto in \cite[Theorem 3.3]{His16b} proved the ``only if" part.


In this paper, we will also consider  the following stronger stability condition.
\begin{defn}\label{def-cJstable}
$(X, L)$ is called uniformly $\cJ^{K_X}$-stable (resp. $\cJ^{K_X}$-semistable) if there exists $\gamma>0$ such that for any test configuration $\phi=(\mcX, \mcL)$, we have:
\begin{equation}\label{eq-cJst}
\mcJ^\NA(\phi) \ge \gamma\; \bfJ^\NA(\phi) \quad (\text{resp. } \ge 0). 
\end{equation}
\end{defn}
\subsection{Geodesic rays}\label{sec-rays}

\subsubsection{Finite energy rays}\label{sec-Archrays}
We refer to the papers \cite{BBEGZ, BBJ18, Dar15, GZ17} for precise meanings of the notations in the following discussion. Denote by $\mcH(L)$ the space of smooth Hermitian metrics on $L$ with K\"{a}hler curvature forms. Fix a smooth reference metric $\psi\in \mcH(L)$.
For any $\vphi\in \mcH(L)$, define
\begin{equation}
\bfE(\vphi):=\bfE_\psi(\vphi)=\qV \frac{1}{n+1} \sum_{k=0}^n\int_X (\vphi-\psi)(\ddc\vphi)^k\wedge (\ddc\psi)^{n-k}.
\end{equation}
For any $\vphi\in \PSH(L)$, define:
\begin{equation}
\bfE(\vphi):=\bfE_\psi(\vphi)=\inf\{ \bfE(\tilde{\vphi}); \tilde{\vphi}\ge \vphi, \tilde{\vphi}\in \mcH(L) \}
\end{equation}
and set:
\begin{eqnarray*}
\cE^1&:=&\cE^1(L)=\{\vphi\in \PSH(L); \bfE_\psi(\vphi)>-\infty\} \\
\cE^1_0&:=&\cE^1_0(L)=\{\vphi\in \cE^1; \bfE_\psi(\vphi)=0\}. 
\end{eqnarray*}
For any $\vphi_i \in \cE^1, i=1,2$, we have the following important functionals:
\begin{eqnarray}
\bfE_{\vphi_1}(\vphi_2)&=&\qV \frac{1}{n+1} \sum_{k=0}^n\int_X (\vphi_2-\vphi_1)(\ddc\vphi_2)^k\wedge (\ddc\vphi_1)^{n-k} \label{eq-EE1}\\
\Lam_{\vphi_1}(\vphi_2)&=&\qV \int_X (\vphi_2-\vphi_1)(\ddc\vphi_1)^n \label{eq-LamE1}\\
\bfJ_{\vphi_1}(\vphi_2)&:=&\Lam_{\vphi_1}(\vphi_2)-\bfE_{\vphi_1}(\vphi_2),\label{eq-defJ}\\
\bfI(\vphi_1, \vphi_2)&=&\qV \int_X (\vphi_2-\vphi_1)\left((\ddc\vphi_1)^n-(\ddc\vphi_2)^n\right). \label{eq-defI} 
\end{eqnarray}
The inequalities in the following lemma will be useful to us. 
\begin{lem}
Assume $\vphi_1, \vphi_2, \vphi_3\in \cE^1(L)$. 
\begin{enumerate}
\item The following inequalities hold true:
\begin{equation}\label{eq-IJineq}
\frac{n}{n+1}\bfI(\vphi_1, \vphi_2) \ge\bfJ_{\vphi_1}(\vphi_2) \ge \frac{1}{n+1}\bfI(\vphi_1, \vphi_2) \ge 0.
\end{equation} 
\begin{equation}\label{eq-Jge0}
\bfE_{\vphi_1}(\vphi_2)\le  \qV \int_X (\vphi_2-\vphi_1)(\ddc\vphi_1)^n.
\end{equation}
\item

$\bfJ$ is convex: for any $t\in [0, 1]$, we have:
\begin{equation}\label{eq-Jconvex}
\bfJ_{\vphi_1}((1-t)\vphi_2+t \vphi_3 )\le (1-t)\bfJ_{\vphi_1}(\vphi_2)+ t\; \bfJ_{\vphi_1}(\vphi_3).
\end{equation}
As a consequence, there exists $C_n>0$ such that for any $t\in [0,1]$
\begin{equation}\label{eq-Iconvex}
\bfI(\vphi_1, (1-t)\vphi_2+t \vphi_3)\le C_n \left((1-t) \bfI(\vphi_1, \vphi_2)+t\; \bfI(\vphi_1, \vphi_3)\right).
\end{equation}
\item
If $\vphi_1\le \vphi_2$, then 
\begin{equation}\label{eq-Isimple2}
\bfI(\vphi_1,\vphi_2)\le \qV \int_X (\vphi_2-\vphi_1)((\ddc\vphi_1)^n+(\ddc\vphi_2)^n)\le (n+1) \bfE_{\vphi_1}(\vphi_2).
\end{equation}
\end{enumerate}
\end{lem}
\begin{proof}[Sketch of proof]
\eqref{eq-IJineq} is well-known (see \cite{Tia00} and \cite{BBGZ13}). \eqref{eq-Jge0} just says that $\bfJ_{\vphi_1}(\vphi_2)\ge 0$. 
\eqref{eq-Jconvex} follows from the concavity of the function $t\mapsto \bfE_{\vphi_1}((1-t)\vphi_2+t \vphi_3)$. \eqref{eq-Iconvex} follows from \eqref{eq-Jconvex}.
\eqref{eq-Isimple2} is immediate.
\end{proof}

In this paper, $C_n$ will denote any constant depending only on the dimension $n$.
We will use the following important estimates 
\begin{prop}
\begin{enumerate}
\item {\cite[Theorem 1.8]{BBEGZ}} $\bfI$ satisfies a quasi-triangle inequality: for any $\vphi_i\in \cE^1, i=1,2,3$, we have:
\begin{equation}\label{eq-Itri}
\bfI(\vphi_1, \vphi_3)\le C_n( \bfI(\vphi_1, \vphi_2)+\bfI(\vphi_2, \vphi_3)).
\end{equation} 
\item {\cite[Appendix A]{BBJ18}} For any $\vphi_i\in \cE^1, 1\le i\le 4$, we have the estimate:
\begin{equation}\label{eq-BBJest}
\int_X (\vphi_2-\vphi_1)((\ddc\vphi_3)^n-(\ddc\vphi_4)^n)\le C_n \bfI(\vphi_2, \vphi_1)^{\frac{1}{2^n}} \bfI(\vphi_3, \vphi_4)^{\frac{1}{2^n}} M^{1-\frac{1}{2^{n-1}}},
\end{equation}
where $M=\max_{1\le i\le 4}\{\bfI(\psi, \vphi_i)\}$.
\end{enumerate}
\end{prop}
By general theory of metric topological spaces, $\bfI$
defines a metrizable structure on $\cE^1_0(L)$ (note that $\bfI$ is translation invariant). Darvas \cite{Dar15} defined a Finsler-type $d_1$-distance on $\cE^1(L)$ and proved that $(\cE^1(L), d_1)$ is the metric completion of $(\mcH(L), d_1)$ whose metric topology coincides with the strong topology introduced in \cite{BBEGZ}.
Moreover, he proved 
\begin{thm}[{\cite[Theorem 3]{Dar15}}]\label{thm-Darv}
There exists a universal constant $C>0$, such that for any $\vphi_1, \vphi_2\in \cE^1(L)$,
\begin{equation}\label{eq-d1toI1}
C^{-1} \bfI_1(\vphi_1, \vphi_2)\le d_1(\vphi_1, \vphi_2)\le C \bfI_1(\vphi_1, \vphi_2).
\end{equation} 
where 
\begin{equation}
\bfI_1(\vphi_1, \vphi_2):=\qV \int_X |\vphi_2-\vphi_1|((\ddc\vphi_1)^n+(\ddc\vphi_2)^n).
\end{equation}
\end{thm}
Note that in general we then have:
\begin{equation}\label{eq-IvsI1}
\bfI(\vphi_1, \vphi_2)\le \bfI_1(\vphi_1, \vphi_2)\le C d_1(\vphi_1, \vphi_2).
\end{equation}

In this note, we use the following notions in \cite{BBJ18} (see \cite[Corollary 1.8]{BBJ18}). A {\it psh ray} (resp. {\it psh path}) is a continuous map $$\Phi=\{\vphi(s)\}: \bR_{\ge 0}\rightarrow \cE^1 \quad (\text{resp. } \Phi: [a,b]\rightarrow \cE^1)$$ such that the $S^1$-invariant Hermitian metric $\Phi=\{\vphi(-\log|t|^2)\}$ on $p_1^*L$ over $X\times\{t\in \bC^*; |t|\le 1\}$ (resp. $X\times \{t\in \bC^*; |t|\in [e^{-b/2}, e^{-a/2}]\})$ is a psh Hermitian metric on $p_1^*L$, i.e. it is locally represented by plurisubharmonic potential functions. 
A (finite energy) geodesic joining $\vphi_i, i=1,2\in \cE^1(L)$ is by definition the largest psh path dominated by $\vphi_0$ and $\vphi_1$.  Any (finite energy) {\it geodesic ray} in this paper is a finite energy geodesic ray emanating from the fixed reference metric $\psi$. We say that $\Phi$ is {\it sup-normalized} if $\sup(\vphi(s)-\psi)=0$ for any $s\in \bR_{\ge 0}$. For the construction of such geodesic rays, we refer to \cite{Dar17a, PS07, RWN14} and references therein.

By the work of Berman-Darvas-Lu and Chen-Cheng, we have the important:
\begin{thm}[{\cite[Proposition 5.1]{BDL17},\cite[Corollary 5.6]{CC18b}}]\label{thm-d1convex}
Let $\Phi_i=\{\vphi_i(s)\}: \bR_{\ge 0}\rightarrow \cE^1, i=1,2$ be two geodesic rays emanating from $\psi$. Then the function $s\mapsto d_1(\vphi_1(s), \vphi_2(s))$ is convex on $[0, \infty)$. As a consequence, the following limit exists, which may be $+\infty$:
\begin{equation}
d^c_1(\Phi_1, \Phi_2):=\lim_{s\rightarrow+\infty}\frac{d_1(\vphi_1(s), \vphi_2(s))}{s}.
\end{equation}
Moreover exactly one of the two alternatives holds: either $d_1^c(\Phi_1,\Phi_2)> 0$ or $\Phi_1=\Phi_2$.
\end{thm}

It is known that $\bfE=\bfE_\psi$ is affine along any geodesic. In particular, $\bfE$ is linear along any geodesic ray (emanating from $\psi$).
For any geodesic ray $\Phi=\{\phi(s)\}$, we will set:
\begin{equation}
\bfE'^\infty(\Phi)=\lim_{s\rightarrow+\infty}\frac{\bfE(\vphi(s))}{s}.
\end{equation}
\begin{lem}\label{lem-altern}
Let $\Phi_i=\{\vphi_i(s)\}, i=1,2$ be two geodesic rays emanating from $\psi$. Assume that $\Phi_2\ge \Phi_1$. Then either $\bfE'^\infty(\Phi_2)>\bfE'^\infty(\Phi_1)$ or $\Phi_1=\Phi_2$.
\end{lem}
\begin{proof}
Because $\bfE(\vphi_i(s))$ is linear in $s$, 
this follows easily from the following domination principle in \cite[Proposition 4.2]{BDL17}: if $\{\vphi_1,\vphi_2\}\subset \cE^1(L)$ satisfies $\vphi_2\ge \vphi_1$, then $\bfE(\vphi_2)\ge \bfE(\vphi_1)$ and the equality holds if and only if $\vphi_1=\vphi_2$. Because $\bfE(\vphi_2)-\bfE(\vphi_1)=\bfE_{\vphi_1}(\vphi_2)\ge C \int_X (\vphi_2-\vphi_1)(\ddc\vphi_2)^n$ by \eqref{eq-Isimple2}, the latter domination principle is reduced to S.Dinew's domination principle in \cite[Proposition 5.9]{BL12} which in turn depends on his uniqueness result. 
\end{proof}
Darvas (\cite{Dar15}) proved that if $\vphi_2\ge \vphi_1$, then $d_1(\vphi_1, \vphi_2)=\bfE(\vphi_2)-\bfE(\vphi_1)$. So the above lemma is a corollary of Theorem \ref{thm-d1convex}. However we state it separately since its proof is in some sense simpler and it is enough for proving Theorem \ref{thm-maximal}.

\subsubsection{Maximal geodesic rays and finite energy non-Archimedean metrics}

For more details of the following definition, we refer to \cite[6]{BBJ18}. 
\begin{defn}
\begin{enumerate}

\item A psh ray $\Psi=\{\psi(s)\}: \bR_{\ge 0}\rightarrow \cE^1$ is of linear growth if 
\begin{equation}
\sup_{s>0}\; s^{-1}\sup(\psi(s)-\psi)<+\infty.
\end{equation}
Any psh ray $\Psi$ of linear growth defines a non-Archimedean metric $\Psi_\NA\in \PSH^\NA(L)$ which is represented by the following function on $X^{\rm div}_\bQ$: for any $v\in X^{\rm div}_\bQ$, let $G(v)$ be its Gauss extension and set
\begin{equation}
\Psi_\NA(v)=-G(v)(\Psi).\quad  (\text{the generic Lelong number w.r.t. } G(v))
\end{equation}
\item
A geodesic ray $\Phi=\{\vphi(s)\}: \bR_{\ge 0}\rightarrow \cE^1$ is maximal if any psh ray of linear growth $\Psi=\{\psi(s)\}: \bR_{\ge 0}\rightarrow \cE^1$ with $\lim_{s\rightarrow 0}\psi(s)\le \vphi(0)$ and 
$\Psi_\NA\le \Phi_\NA$ satisfies $\Psi\le \Phi$.
\end{enumerate}

\end{defn}
The following result says that there is a one-to-one correspondence between $\mcE^{1,\NA}$ and the set of maximal geodesic rays emanating from the fixed reference metric $\psi$.
\begin{thm}[{\cite[6]{BBJ18}}]\label{thm-BBJ}
\begin{enumerate}
\item
For any psh ray $\Phi: \bR_{\ge 0}\rightarrow \cE^1$ of linear growth, the associated non-Archimedean metric $\Phi_\NA$ belongs to $\cE^{1,\NA}(L)$ (see Definition \ref{def-cE1NA}), and
\begin{equation}\label{eq-NAvSlope}
\bfE^\NA(\Phi_\NA)\ge \bfE'^\infty(\Phi)>-\infty.
\end{equation}
Any geodesic ray is of linear growth and hence defines a finite energy non-Archimedean metric.
\item
For any $\phi\in \cE^{1,\NA}$, there exists a unique maximal geodesic ray $\Phi: \bR_{\ge 0}\rightarrow \cE^1$ emanating from $\psi$ satisfying $\Phi_\NA=\phi$. We will also denote this maximal geodesic ray by $\rho^\phi$. In particular, for any test configuration there is a unique maximal geodesic ray emanating from $\psi$ which coincides with the geodesic ray constructed by Phong-Sturm (in \cite{PS07}).
\item 
A geodesic ray $\Phi=\{\vphi(s)\}: \bR_{\ge 0}\rightarrow \cE^1$ (emanating from $\psi$) is maximal if and only if equality holds in \eqref{eq-NAvSlope}, or equivalently $\bfE(\vphi(s))=s \cdot \bfE^\NA(\Phi_\NA)$. Moreover, in this case, let $\{\phi_m\}\subset \mcH^\NA(L)$ be any decreasing sequence converging to $\Phi_\NA$, the following limit identity holds true:
\begin{equation}
\bfE'^\infty(\Phi)=\lim_{m\rightarrow+\infty}\bfE^\NA(\phi_m).
\end{equation}

\end{enumerate}
\end{thm}
See the recent preprint in \cite{DX20} for more equivalent characterizations of maximal geodesic rays.
The proof of Theorem \ref{thm-BBJ} hinges on the following important construction by Berman-Boucksom-Jonsson, which in particular shows that any $\phi\in \cE^{1,\NA}(L)$ can be approximated by a decreasing sequence $\{\phi_m\}\subset \mcH^\NA$. 
We call refer to such a construction by the name of Multiplier Approximation. This kind of construction goes back to \cite{DEL00} and also in the study of non-Archimedean Monge-Amp\`{e}re equations in \cite{BFJ15, BFJ09}.
\vskip 1mm

\noindent
{\bf Multiplier Approximation:}
Let $\Phi:\bR_{\ge 0}\rightarrow \cE^1$ be a geodesic ray such that $\sup(\vphi(t)-\psi)=0$. We extend $\Phi$ to be a singular positively curved Hermitian metric on $p_1^*L\rightarrow X\times\bC$.
Denote by $\mcJ(m\Phi)$ the multiplier ideal sheaf of $m\Phi$. Let $\mu_m: \mcX_m\rightarrow X\times\bC$ be the normalized blow-up of $X\times \bC$ along $\mcJ(m\Phi)$ with exceptional divisor denoted by $E_m$. Set $\mcL_m=\mu_m^*p_1^*L-\frac{1}{m+m_0}E_m$. Then by Castelnuovo-Mumford criterion and Nadel vanishing, for $m_0$ sufficiently large and any $m\gg 1$, $(\mcX_m, \mcL_m)$ is a semiample test configuration of $(X, L)$. Let $\Phi_m$ be the geodesic ray associated to $(\mcX_m, \mcL_m)$ (as constructed in \cite{PS07} or in \cite[Theorem 6.6]{BBJ18}). Then we have 
$\Phi_m\ge \Phi, \quad \Phi_{m,\NA}\ge \Phi_\NA$ and $\lim_{m\rightarrow+\infty}\Phi_{m,\NA}=\Phi_\NA$.

Moreover, we can get a decreasing approximating sequence by setting 
$\check{\Phi}_m=\Phi_{2^m}$. Then $\check{\Phi}_{m,\NA}$ decreases to $\Phi_\NA$ and $\check{\Phi}_m$ decreases to the maximal geodesic ray $\rho^{\Phi_\NA}$. Because of this minor change of subscripts, by abuse of notations, we will also write $\Phi_m$ (resp. $\phi_m:=\Phi_{m,\NA}$) for $\check{\Phi}_m$ (resp. $\check{\Phi}_{m,\NA}$).

If $\Phi$ is not sup-normalized, then we know that $\sup(\vphi(s)-\psi)=c s$ for some $c\in \bR$ by \cite[Proposition 1.10]{BBJ18}. We set $\Phi':=\Phi+c\log|t|^2$ and apply Multiplier Approximation to $\Phi'$ and get $\Phi'_m$. Then we set $\Phi_m=\Phi_m+c \log|t|^2$.



\subsection{More functionals}\label{sec-Echi}
For any smooth closed (1,1)-form $\chi$ and $\vphi_1, \vphi_2\in \cE^1(L)$, define
\begin{equation}\label{eq-Echi}
\bfE^\chi_{\vphi_1}(\vphi_2)=\qV \int_X (\vphi_2-\vphi_1)\chi\wedge \sum_{k=0}^{n-1}(\ddc\vphi_2)^k\wedge (\ddc\vphi_1)^{n-1-k}.
\end{equation}
We write $\bfE^\chi(\vphi)=\bfE_{\psi}^\chi(\vphi)$ for any $\vphi\in \cE^1(L)$. An easy calculation shows the identity:
\begin{equation}
\bfI(\vphi_1, \vphi_2)=\bfE^{\ddc \vphi_1}_{\vphi_1}(\vphi_2)+\Lam_{\vphi_1}(\vphi_2).
\end{equation}
Moreover if $[\chi']=[\chi]$, then by the $\partial\bar{\partial}$-lemma $\chi'-\chi=\ddc f$ for some $f\in C^{\infty}(X)$ and for any $\vphi_1, \vphi_2\in \mcE^1(L)$ we have the identity:
\begin{equation}\label{eq-E2chi}
\bfE^{\chi'}_{\vphi_1}(\vphi_2)-\bfE^{\chi}_{\vphi_1}(\vphi_2)=\qV  \int_X f \left[(\ddc\vphi_2)^n-(\ddc\vphi_1)^n\right].
\end{equation}

\begin{lem}\label{lem-Echislope}
Let $\Phi: \bR_{\ge 0}\rightarrow \cE^1$ be a geodesic ray. Then the limit 
\begin{equation}
\lim_{s\rightarrow+\infty} \frac{\bfE^\chi(\vphi(s))}{s}=: (\bfE^\chi)'^\infty(\Phi)
\end{equation}
exists and is finite.
\end{lem}
\begin{proof}
Any smooth $(1,1)$-form $\chi$ can written as $\chi=\chi_1-\chi_2$ with $\chi_i, i=1,2$ being smooth K\"{a}hler forms. So we can assume $\chi$ is smooth and K\"{a}hler. 
Then we know that $\bfE^\chi(\vphi(s))$ is convex with respect to $s$. Indeed, this is well known if $\Phi$ is $C^{1,\bar{1}}$ (see \cite[4.1, Proposition 2]{Che00}). In general, we can approximate any segment of $\Phi$ by $C^{1,\bar{1}}$-geodesic segments and use the continuity of $\bfE^\chi$ under $d_1$-convergence (see \cite[Lemma 5.23]{DR17}) to get the convexity.  Moreover because $\Phi$ has linear growth, $\bfE^\chi(\vphi(s))$ also grows at most linearly with respect to $s$. So the slope $(\bfE^\chi)'^\infty(\Phi)$ indeed exists and is finite. 
\end{proof}
The constant scalar curvature K\"{a}hler metrics are global minimizers of Mabuchi functional $\bfM(\vphi)$.
We now have the following important analytic criterion for the existence of cscK metrics:
\begin{thm}[{\cite{CC18a, CC18b}, \cite{DR17, His16b}}]
$(X, L)$ admits a cscK metric if and only if the Mabuchi functional is $\Aut(X, L)_0$-coercive. 
\end{thm}

 We will use the following Chen-Tian's formula for Mabuchi energy and other related functionals. 
\begin{defn}
For any smooth volume forms $\Omega$ and $\nu$, and any $\vphi\in \cE^1(L)$, define:
\begin{eqnarray}
\bfH_\Omega(\vphi)&:=&\qV \int_X\log \frac{(\ddc\vphi)^n}{\Omega}(\ddc\vphi)^n; \label{eq-HOmega} \\
\bfR(\vphi)&:=&\bfE^{-Ric(\Omega)}(\vphi)\\
\bfS(\vphi)&:=&\bfS_\psi(\vphi)=\bfH_{\Omega}(\vphi)-\bfH_{\Omega}(\psi)+\bfE^{-Ric(\Omega)}_\psi(\vphi)\\
\cJ(\vphi)&:=&\mcJ^{-Ric(\Omega)}(\vphi)=\bfE^{-Ric(\Omega)}(\vphi)+\ud{S}\bfE(\vphi) \label{eq-cJ}\\
\bfM(\vphi)&:=&\bfM_\psi(\vphi)=\bfH_\Omega(\vphi)-\bfH_\Omega(\psi)+\bfE^{-Ric(\Omega)}(\vphi)+\ud{S}\bfE(\vphi)\nonumber  \\
&=&\bfS(\vphi)+\ud{S}\bfE(\vphi)=\bfH_\Omega(\vphi)-\bfH_\Omega(\psi)+\cJ(\vphi)\nonumber \\
&=&\qV \int_X \log \frac{(\ddc\vphi)^n}{\Omega}(\ddc\vphi)^n-\qV \int_X \log\frac{(\ddc\psi)^n}{\Omega}(\ddc\psi)^n\nonumber \\
&&\hskip 1cm+\bfE^{-Ric(\Omega)}_\psi(\vphi)+\udS \bfE_\psi(\vphi) \label{eq-Mabuchi}\\
\bfS(\vphi; \nu)&:=&\qV \int_X \log \frac{\nu}{\Omega}(\ddc\vphi)^n+\bfE^{-Ric(\Omega)}_\psi(\vphi)-\bfH_\Omega(\psi) \label{eq-Gphinu}.
\end{eqnarray}
\end{defn}
In the above formula, by replacing $\psi$, $\vphi$ by $\vphi_1$, $\vphi_2$ respectively, we define $\bfS_{\vphi_1}(\vphi_2)$ and $\bfM_{\vphi_1}(\vphi_2)$ and so on.
It is easy to verify that the functionals $\bfS_{\vphi_1}(\vphi_2)$ and $\bfM_{\vphi_1}(\vphi_2)$ depend only on $\vphi_i, i=1,2$  but  not on the volume form $\Omega$.
We will use the following simple but important co-cycle property: for any $\vphi_i\in \cE^1, i=1,2,3$, and for $\bfF\in \{\bfE, \bfE^\chi, \bfS, \bfM\}$, 
\begin{equation}\label{eq-cocycle}
\bfF_{\vphi_1}(\vphi_3)=\bfF_{\vphi_1}(\vphi_2)+\bfF_{\vphi_2}(\vphi_3).
\end{equation}
This is well-known and can be verified by using integration by parts if $\vphi_i$ are smooth. For general $\vphi_i$, it can be verified using approximation argument.

The following result connects the Archimedean and non-Archimedean functionals :
\begin{prop}[see \cite{BHJ17, BHJ19} and references therein]\label{prop-ARvsNA}
With the notations as above, for any $\phi=(\mcX, \mcL)\in \mcH^\NA(L)$, let $\Phi$ be any locally bounded $S^1$-invariant Hermitian metric on $\mcL$. 
Then for any $\bfF\in \{\bfE, \Lam, \bfJ,  \bfI,  \cJ, \bfR\}$, we have the identity:
\begin{equation}\label{eq-slope1}
\bfF'^\infty(\Phi)=\bfF^\NA(\phi).
\end{equation}
In particular, the identity is true if $\Phi$ is the maximal geodesic ray associated to $\phi$. 

Moreover, for $\bfF\in \{\bfH_\Omega, \bfM\}$, the identity \eqref{eq-slope1} holds true if $\Phi$ is a smooth positively curved Hermitian metric on $\mcL$.
\end{prop}
 As an application of Proposition \ref{prop-ARvsNA}, we get an estimate of the non-Archimedean entropy.
First recall Tian's $\alpha$-invariant:
\begin{eqnarray*}
\alpha(X, L)&=&\sup\left\{\alpha>0; \int_X e^{\alpha\left(\sup(\vphi-\psi)-(\vphi-\psi)\right)}\omega^n<C<+\infty\quad \forall \vphi\in \PSH(L)\right\}\\
&=&\inf\; \left\{ \lct(X, D); m D\in |mL| \text{ for some } m\in \bN\right\}.
\end{eqnarray*}
The following lemma is well-known.
\begin{lem}[{\cite[Lemma 3.17]{Der16}, \cite[Theorem 1.6]{DK19}, \cite[Proposition 9.16]{BHJ17}}, \cite{Oda12}]\label{lem-J2Kst}
We have the inequality $\bfH^\NA(\phi)\ge \alpha(X, L) \bfI^\NA(\phi)$ for any $\phi\in \mcH^\NA(L)$. As a consequence,
if $(X, L)$ is $\cJ^{K_X}$-semistable then $(X, L)$ is uniformly K-stable.
\end{lem}
\begin{proof}
For any test configuration $(\mcX, \mcL)$, let $\Phi=\{\vphi(s)\}$ be a smooth positively curved Hermitian metric on $\mcL$. 
Following \cite{Tia00}, for any $\epsilon>0$ by Jensen's inequality we get the inequality:
\begin{eqnarray*}
\int_X \log\frac{(\ddc\vphi)^n}{(\ddc\psi)^n}(\ddc\vphi)^n&\ge& (\alpha-\epsilon) \left(\sup(\vphi-\psi)-\qV \int_X (\vphi-\psi)(\ddc\vphi)^n\right)-C\\
&\ge& (\alpha-\epsilon)\left( \qV \int_X (\vphi-\psi)((\ddc\psi)^n-(\ddc\vphi)^n)\right)-C \\
&=& (\alpha-\epsilon) \bfI(\psi,\vphi)-C.
\end{eqnarray*}
Taking slopes on both sides and using Proposition \ref{prop-ARvsNA} we get $\bfH^\NA(\phi)\ge (\alpha-\epsilon) \bfI^\NA(\phi)$. Letting $\epsilon\rightarrow 0$, we get the conclusion.
\end{proof}

The identity \eqref{eq-slope1} can be proved using the Deligne pairing (see \cite[Lemma 3.9]{BHJ19} or \cite[Lemma 6]{PRS08}). The same method of proof gives the following result.
\begin{prop}\label{prop-Echislope2}
Consider the data:
\begin{enumerate}
\item[(1)] 
Let $(\mcX, \mcL)$ be a test configuration of $(X, L)$ and $\Phi=\{\vphi(s)\}$ be a locally bounded $S^1$-invariant Hermitian metric on $\mcL$. 
\item[(2)]  $\pi_{\mcY}: \mcY\rightarrow \bC$ is a model of $X$ (see Definition \ref{def-TC}). Let $\mcQ$ be a $\bC^*$-equivariant $\bQ$-line bundle over $\mcY$, and $\Psi_\mcQ=\{\psi_\mcQ(s)\}$ be a smooth $S^1$-invariant Hermitian metric on $\mcQ$. 
\end{enumerate}
Then we have the slope formula
\begin{equation}\label{eq-EchiAR2NA}
\lim_{s\rightarrow+\infty} \frac{\bfE^{\ddc\psi_\mcQ (s)}_\psi(\vphi(s))}{s}=(\bfE^\cQ)^\NA(\phi_{(\mcX, \mcL)})
\end{equation}
where $(\bfE^\cQ)^\NA$ was defined in \eqref{eq-EcQNA} or \eqref{eq-intEcQNA}. In particular, we have:
\begin{equation}
\lim_{s\rightarrow+\infty} \frac{\bfE^{-Ric(\omega)}(\vphi(s))}{s}=\qV {K^{\log}_{X_{\bP^1}/\bP^1}}\cdot \bar{\mcL}^n=\bfR^\NA(\phi_{(\mcX, \mcL)}).
\end{equation}
\end{prop}
Indeed, this follows from the fact that
$\bfE^{\ddc\psi_\mcQ(s)}(\vphi(s))$ can be considered as an Hermitian metric on the tensor product of line bundles obtained by using Deligne pairings:
\begin{equation}
\la \mcQ, \overbrace{\mcL, \cdots, \mcL}^{n}\ra\otimes \la \mcQ, \overbrace{\pi^*L, \cdots, \pi^*L}^n\ra^{-1}.
\end{equation}

The following lemma says that for any maximal geodesic ray $\Phi$, the identity \eqref{eq-slope1} holds true for the $\Lam$ functional.
\begin{lem}\label{lem-energyslope}
For any $\phi\in \cE^{1,\NA}(L)$, let $\Phi=\{\vphi(s)\}$ be the associated maximal geodesic ray. Then we have the identity:
\begin{equation}\label{eq-Lamslope}
\Lam^\NA(\phi)=\sup(\phi-\phi_{\triv})=\lim_{s\rightarrow+\infty} \frac{\sup(\vphi(s)-\psi)}{s}=\Lam'^\infty(\Phi).
\end{equation}
As a consequence we have the identity:
\begin{equation}
\bfJ^\NA(\phi)=\bfJ'^\infty(\Phi).
\end{equation}

\end{lem}
\begin{proof}
If $\phi\in \mcH^\NA(L)$, \eqref{eq-Lamslope} holds true by the work \cite{BHJ19, BBJ18}. For general $\phi$, we can use Multiplier Approximation to find $\{\phi_m\}\subset \mcH^\NA(\phi)$, which decreases to $\phi$ and the corresponding $\Phi_m$ decreases to $\Phi$. As $m\rightarrow+\infty$, $\sup(\phi_m-\phi_{\triv})$ converges to $\sup(\phi-\phi_{\triv})$ by \cite[Theorem 7.10]{BFJ16}. Moreover $\sup(\vphi_m(s)-\psi)$ is affine and converges to $\sup(\vphi-\psi)$ as $m\rightarrow+\infty$. So \eqref{eq-Lamslope} holds true for $\phi$ by letting $m\rightarrow+\infty$
\end{proof}
We will prove in Theorem \ref{thm-grslope}.2 that identity \eqref{eq-slope1} holds true for $\bfF\in\{\bfH, \bfM\}$ if $\Phi$ is the (maximal) geodesic ray associated to any test configuration. We summarize the results about the identity $\bfF'^\infty=\bfF^\NA$ for various functionals in the following table.
\begin{center}
\begin{tabular}{ |l|l|l| } 
 \hline
Truth of & \multirow{2}{*}{Metric $\Phi$ on $\mcL\rightarrow \mcX$} & Maximal Geodesic Ray \\
$\bfF'^\infty=\bfF^\NA$ & & (including destabilizing geodesic rays) \\
  \hline
$\bfE, \Lam, \bfJ$ & locally bounded (Proposition \ref{prop-ARvsNA}) & Theorem \ref{thm-BBJ}, Lemma \ref{lem-energyslope}  \\
\hline
$\bfE^\chi$, $\bfI$ & locally bounded (Proposition \ref{prop-Echislope2}) & Theorem \ref{thm-Echislope} \\ 
\hline 
\multirow{2}{*}{$\bfH$, $\bfM$} &  smooth positive metric (\cite{Tia97, Tia14}, \cite{BHJ19}) & $``\ge"$ (Theorem \ref{thm-grslope}.1) \\ 
& associated geodesic ray (Theorem \ref{thm-grslope}.2) &  $``\le "$ (open, Conjecture \ref{conj-Hslope})   \\ 
 \hline
\end{tabular}
\end{center}

\vskip 2mm

\section{Geodesic rays with finite Mabuchi slopes}\label{sec-maximal}
Let $\Omega$ be a fixed smooth volume form on $X$ satisfying $\int_X \Omega=\int_X (\ddc\vphi)^n= V$. Denote by $\bB=\{t\in \bC; |t|\le 1\}$ the unit disc in $\bC$. We will interchangeably use the variable $t$ to denote coordinate on $\bC$ and the variable $s=-\log|t|^2\in [0, +\infty)$ for any $t\in \bB\setminus \{0\}$. 

The following equisingular lemma is similar to \cite[Lemma 2]{Dem15}.
\begin{lem}\label{lem-integrable}
Let $\Phi: \bR_{\ge 0}\rightarrow \cE^1$ be a finite energy geodesic ray. Let $\hat{\Phi}=\rho^{\Phi_\NA}$ be the associated maximal geodesic ray.  Then 
 for any $\alpha>0$, we have:
\begin{equation}\label{eq-intfinite}
\int_{X\times\bB}e^{\alpha(\hat{\Phi}-{\Phi})} \Omega\wedge \sqrt{-1}dt\wedge d\bar{t}<+\infty. 
\end{equation}
\end{lem}
\begin{proof}
By \cite[Proposition 1.10]{BBJ18}, $s\mapsto \sup(\vphi(s)-\psi)$ is affine. So by subtracting an affine function, we can assume $\sup(\vphi(s)-\psi)=0$. We then carry out Multiplier Approximation (see the end of section \ref{sec-rays}) to get test configurations $(\mcX_m, \mcL_m)$. Let $\Phi_m$ be the geodesic ray associated to $(\mcX_m, \mcL_m)$.  
Locally the singularity of $\Phi_m$ is comparable to $\frac{1}{m+m_0}\log\sum_i |f_i|^2$ where $\{f_i\}$ are generators of $\mcJ(m\Phi)$.
By Demailly's regularization theorem which depends on the Ohsawa-Takegoshi extension theorem, locally $(m+m_0)\Phi_m$ is less singular than $m \Phi$. 
Let $\Psi_\triv=p_1^*\psi$ be the trivial geodesic ray. Then there exists $C=C_m$ such that:
\begin{equation}
(m+m_0)\Phi_m\ge m \Phi+m_0 \Psi_{\triv}-C.
\end{equation}
Because $\Phi_\NA=\hat{\Phi}_\NA$, we have $\mcJ(m\Phi)=\mcJ(m\hat{\Phi})$ by the valuative description of multiplier ideal sheaves (\cite[Theorem 5.5]{BFJ08}). So similarly we can assume that $$(m+m_0)\Phi_m\ge m \hat{\Phi}+m_0\Psi_{\triv}-C.$$
Then we have:
\begin{eqnarray*}
e^{\alpha(\hat{\Phi}-{\Phi})}&=&e^{\alpha(\hat{\Phi}+\frac{m_0}{m}\Psi_{\triv}-\frac{m+m_0}{m}\Phi_m)} e^{\alpha(\frac{m+m_0}{m}\Phi_m-\Phi-\frac{m_0}{m}\Psi_\triv)}\\
&\le& e^{\frac{\alpha C}{m}} e^{\alpha(\frac{m+m_0}{m}\Phi_m-\Phi-\frac{m_0}{m}\Psi_{\triv})}.
\end{eqnarray*}

When $\alpha=m$, we get:
\begin{equation}
e^{m(\hat{\Phi}-{\Phi})}\le e^{C} e^{(m+m_0)\Phi_m-m\Phi-m_0 \Psi_{\triv}}.
\end{equation}
This is integrable by the definition of multiplier ideal sheaf $\mcJ(m\Phi)$ (see \cite[Proof of Lemma 2]{Dem15}). Since $m$ can be arbitrarily big, we get the conclusion.
\end{proof}
\begin{rem}
As pointed out by Boucksom, this lemma is essentially a local result about psh functions and related questions have been extensively studied in \cite{BFJ08, FJ05}.
\end{rem}


\begin{proof}[Proof of Theorem \ref{thm-maximal}]
Recall that we have the formula:
$$\bfM(\vphi)=\bfM_\psi(\vphi)=\bfH_{\Omega}(\vphi)-\bfH_{\Omega}(\psi)+\bfE^{-Ric(\Omega)}_\psi(\vphi)+\ud{S}\bfE_\psi(\vphi).$$ 
By Lemma \ref{lem-Echislope}, we know that the slope
$(\bfE^{-Ric(\Omega)})'^\infty(\Phi)$ exists and is finite.
Moreover $\bfE'^\infty(\Phi)$ is also finite (since $\Phi$ has a linear growth). So we know that $\bfH'^\infty$ exists and is finite.
So we just need to show that if $\Phi$ is not maximal, then $\bfH'^\infty_\Omega(\Phi)$ is arbitrary large, which would contradict the assumption of finite Mabuchi slope. 

In general we have $\Phi\le \hat{\Phi}$. So $U=\hat{\Phi}-\Phi=\{u(s)=\hat{\vphi}(s)-\vphi(s)\}$ satisfies $U\ge 0$. By Lemma \ref{lem-integrable}, for any $\alpha>0$ we have:
\begin{equation}
\int_{X\times \bB} e^{\alpha U}\Omega\wedge \sqrt{-1} dt\wedge \bar{dt}<+\infty.
\end{equation}
By the $S^1$-invariance, we use the variable $s=-\log|t|^2$ to re-write the integral as
\begin{equation}
\int_0^{+\infty}\left(\int_X e^{\alpha u(s)}\Omega\right) e^{-s}ds<+\infty.
\end{equation}
So there exist $s_j\rightarrow+\infty$ such that
\begin{equation}
e^{-s_j}\int_X e^{\alpha u(s_j)}\Omega \rightarrow 0.
\end{equation}
So we can assume that 
\begin{equation}
\qV \int_X e^{\alpha u(s_j)}\Omega\le e^{s_j}.
\end{equation}
Re-write the above inequality as:
\begin{equation}
\qV \int_X e^{\alpha u(s_j)-\log \frac{(\ddc\vphi(s_j))^n}{\Omega}}(\ddc\vphi(s_j))^n\le e^{s_j}.
\end{equation}
By Jensen's inequality for the probability measure $V^{-1} (\ddc\vphi(s_j))^n$, we get:
\begin{eqnarray}\label{eq-jensen1}
&&\qV \int_X \log\frac{(\ddc\vphi(s_j))^n}{\Omega}(\ddc\vphi(s_j))^n\nonumber \\
&&\hskip 2cm \ge \alpha \cdot\left[\qV \int_{X}(\hat{\vphi}(s_j)-{\vphi}(s_j))(\ddc\vphi(s_j))^n\right]-s_j V. 
\end{eqnarray}
The above inequality is valid when $\ddc\vphi(s_j)$ is a smooth K\"{a}hler form. We claim that it is still true when $\vphi(s_j)\in \cE^1$. To see this, we first use the fact that the entropy of $\ddc\vphi(s_j)$ can be approximated by entropies of smooth K\"{a}hler forms (\cite[Lemma 3.1]{BDL17}). More precisely, there exist $\vphi_k\in \mcH(L)$ such that
\begin{equation}\label{eq-jensen1b}
\bfH_\Omega(\vphi_k)\rightarrow \bfH_\Omega(\vphi(s_j)), \quad \bfI(\vphi_k, \vphi(s_j))\rightarrow 0.
\end{equation} 
Then the same application of Jensen's inequality gives:
\begin{equation}
\bfH_\Omega(\vphi_k)\ge \alpha\left[\qV \int_X (\hat{\vphi}(s_j)-{\vphi}(s_j))(\ddc\vphi_k)^n\right]-s_j V.
\end{equation}
As $k\rightarrow+\infty$, the left hand side converges to $\bfH_\Omega(\vphi)$. By \cite[Proposition 5.6]{BBGZ13}, $\int_X u (\ddc\vphi_k)^n\rightarrow \int_X u (\ddc\vphi(s_j))^n$ uniformly with respect to $u\in \cE^1_C=\{\psi+u \in \cE^1; \bfE_\psi(\psi+u)\ge -C\}$ for fixed $C$. It is then easy to 
show that the right-hand-side of \eqref{eq-jensen1b} indeed converges to the right-hand-side of \eqref{eq-jensen1} (for fixed $s_j$). So the claim follows.

To continue the estimate of the right-hand-side of \eqref{eq-jensen1}, we use \eqref{eq-Jge0} to get:
\begin{eqnarray}\label{eq-ent2dist}
\bfH_\Omega(\vphi)
&\ge& \alpha\bfE_{\vphi(s_j)}(\hat{\vphi}(s_j))-s_j\nonumber \\
&=&\alpha\left(\bfE_{\psi}(\hat{\vphi}(s_j))-\bfE_{\psi}(\vphi(s_j))\right)-s_j V.
\end{eqnarray}
Now dividing both sides by $s_j$ and letting $s_j\rightarrow+\infty$, we get:
\begin{eqnarray*}
\bfH'^\infty(\Phi)\ge \alpha\cdot (\bfE'^\infty(\hat{\Phi})-\bfE'^\infty(\Phi))-V.
\end{eqnarray*}
If $\Phi$ is not maximal, then ${\Phi}\neq {\hat{\Phi}}$ but $\hat{\Phi}\ge \Phi$. Hence by Lemma \ref{lem-altern}, $\bfE'^\infty(\hat{\Phi})-\bfE'^\infty(\Phi)$ is positive. 
But $\alpha$ can be arbitrary large. This is impossible if $\bfH'^\infty(\Phi)$ is finite.
\end{proof}

\section{Convergence of twisted Monge-Amp\`{e}re slopes}\label{sec-twistMA}


Choose a sequence $\{\phi_m\}\subset \mcH^\NA(L)$ that decreases to $\phi$. 
By Theorem \ref{prop-strongcont}, $(\bfE^{Q_\bC})^\NA$ is continuous for decreasing converging sequences. So we have the identity:
\begin{equation}
\lim_{m\rightarrow+\infty} (\bfE^{Q_\bC})^\NA(\phi_m)=(\bfE^{Q_\bC})^\NA(\phi).
\end{equation}
For simplicity of notations, set $\chi=\ddc \psi_Q$. 
If $\Phi_m=\{\vphi_m(s)\}$ is the maximal geodesic ray associated to $\phi_m$, then by the identity \eqref{eq-EchiAR2NA}, we have:
\begin{equation}
 (\bfE^\chi)'^\infty(\Phi_m)=\lim_{s\rightarrow+\infty}\frac{\bfE^\chi(\vphi_m(s))}{s}=(\bfE^{Q_\bC})^\NA(\phi_m).
\end{equation}
Combining the above two identities, proving \eqref{eq-EQslope} reduces to proving:
\begin{equation}\label{eq-Echiconv}
\lim_{m\rightarrow+\infty}(\bfE^\chi)'^\infty(\Phi_m)=(\bfE^\chi)'^\infty(\Phi).
\end{equation}
By writing $\chi$ as $\chi_2-\chi_1$, we can assume that $\chi$ is K\"{a}hler so that $\bfE^\chi$ is monotone increasing. Because $\Phi_m\ge \Phi$, it is immediate that $(\bfE^\chi)'^\infty(\Phi_m)\ge (\bfE^\chi)'^\infty(\Phi)$. 
On the other hand, there exists a constant $C>0$ such that $\chi\le C (\ddc\psi)$. For simplicity of notation, we denote the K\"{a}hler form $\ddc\psi$ by $\omega$. By using the co-cycle condition \eqref{eq-cocycle} and $\Phi_m\ge \Phi$, we easily get:
\begin{equation}
0\le (\bfE^\chi)'^\infty(\Phi_m)-(\bfE^\chi)'^\infty(\Phi)\le C\left[ (\bfE^\omega)'^\infty(\Phi_m)-(\bfE^\omega)'^\infty(\Phi)\right].
\end{equation}
If the right-hand-side converges to 0 as $m\rightarrow +\infty$, so does the left hand side. 
So from now on we just assume that $\chi=\ddc\psi=\omega$. To prove the convergence of slopes, it suffices to show the following estimate: there exists a sequence $\delta_m\rightarrow 0$ as $m\rightarrow+\infty$ such that for any $s>0$,
\begin{equation}\label{eq-Ephimdecay}
\bfE^\omega_\psi(\vphi_m(s))\le \bfE^\omega_\psi(\vphi)+ \delta_m s.
\end{equation} 
In order to estimate the quantity:
\begin{eqnarray*}
&&\bfE^\omega_\psi(\vphi_m)-\bfE^\omega_\psi(\vphi)=\bfE^\omega_{\vphi}(\vphi_m)\\
&=&\qV \sum_{k=0}^{n-1} \int_X (\vphi_m-\vphi)(\ddc\psi)\wedge (\ddc\vphi_m)^k\wedge (\ddc\vphi)^{n-1-k},
\end{eqnarray*}
we will adapt the idea of proof of $d_1$-continuity for $\bfE^\omega$ in \cite[Lemma 5.23]{DR17} which in turn depends on the estimate from \cite[Lemma 5.8]{BBGZ13}. Here we will use directly the refinement of this estimate as proved in \cite[Appendix A]{BBJ18}, 
which allows us to use the simpler $\bfI$ distance to carry out our argument \footnote{In the 1st version the author used an original estimate from \cite{BBGZ13}, the simplification via \cite[Appendix A]{BBJ18} is suggested by S. Boucksom.}. 
First note that we have the identity: 
\begin{equation}\label{eq-d1phim}
\bfE(\vphi_m(s))-\bfE(\vphi(s))=(\bfE'^\infty(\Phi_m)-\bfE'^\infty(\Phi))s=:\epsilon_ms.
\end{equation}
Because $\Phi$ is maximal, the linear functions $\bfE(\vphi_m(s))$ decrease to the linear function $\bfE(\vphi(s))$ as $m\rightarrow+\infty$. So we get:
\begin{equation}
\lim_{m\rightarrow+\infty} \epsilon_m= 0.
\end{equation}


Now set $\tilde{\vphi}_m=\frac{1}{3}\left(\psi+\vphi_m+\vphi\right)$. Because $\vphi_m\ge \vphi$, we easily get:
\begin{eqnarray}\label{eq-Ephiphim}
\bfE^\omega_\vphi(\vphi_m)&\le& C_n \int_X (\vphi_m-\vphi)(\ddc\tilde{\vphi}_m)^n\nonumber \\
&=&C_n \int_X (\vphi_m-\vphi)((\ddc\tilde{\vphi}_m)^n-(\ddc\vphi)^n)+C_n\int_X (\vphi_m-\vphi)(\ddc\vphi)^n\nonumber \\
&=: &\mathfrak{A}+C_n \int_X (\vphi_m-\vphi)(\ddc\vphi)^n.
\end{eqnarray}
Because $\vphi_m\ge \vphi$, the second term is easy to estimate by using \eqref{eq-Isimple2} and \eqref{eq-d1phim}:
\begin{equation}\label{eq-2ndterm}
\bfI(\vphi_m, \vphi)\le \int_X (\vphi_m-\vphi)(\ddc\vphi)^n \le \bfE_\vphi(\vphi_m)\le \epsilon_m s.
\end{equation}
To estimate the first term, we use \eqref{eq-BBJest} to get:
\begin{equation}\label{eq-estA}
\mathfrak{A}\le C_n \bfI(\vphi, \vphi_m)^{\frac{1}{2^n}} \bfI(\vphi, \tilde{\vphi}_m)^{\frac{1}{2^n}} M^{1-\frac{1}{2^{n-1}}},
\end{equation}
where $M=\max\{\bfI(\psi, \vphi_m), \bfI(\psi, \vphi), \bfI(\psi, \tilde{\vphi}_m), \bfI(\psi, \vphi)\}$. Now we have the estimate:
\begin{eqnarray}
\bfI(\psi, \vphi)&\le& C \bfJ_\psi(\vphi)\le C\left(\sup(\vphi-\psi)-\bfE(\psi, \vphi)\right)\le Cs \label{eq-Ilinear}\\
\bfI(\psi, \vphi_m)&\le& C(\bfI(\psi, \vphi)+\bfI(\vphi, \vphi_m))\le C(Cs+\epsilon_m s)\le Cs \nonumber \\
\bfI(\psi, \tilde{\vphi}_m)&=& \bfI\left(\psi, \frac{1}{3}(\psi+\vphi_m+\vphi)\right) \le  C(\bfI(\psi, \vphi_m)+\bfI(\psi, \vphi))\le Cs \nonumber \\
\bfI(\vphi, \tilde{\vphi}_m)&\le& C (\bfI(\psi, \vphi)+\bfI(\psi, \tilde{\vphi}_m))\le C s. \nonumber
\end{eqnarray}
Here we used the inequality \eqref{eq-2ndterm}, quasi-triangle inequality \eqref{eq-Itri} and the quasi-convexity estimate \eqref{eq-Iconvex}.
Plugging these estimates into \eqref{eq-estA}, we get:
\begin{equation}\label{eq-estA2}
\mathfrak{A}\le C \epsilon_m^{\frac{1}{2^n}}s.
\end{equation}
By \eqref{eq-Ephiphim}, this together with \eqref{eq-2ndterm} verifies \eqref{eq-Ephimdecay} and hence finishes the proof of the wanted convergence \eqref{eq-Echiconv}.

By extending the above proof, one can prove the following result which will be used in the next section. 
\begin{thm}\label{thm-EcQslopeconv}
Assume $\phi\in \cE^{1,\NA}(L)$ and let $\Phi$ be the associated maximal geodesic ray emanating from $\psi$.  Let $\pi_{\mcY}: \mcY\rightarrow \bC$ is a model of $X$ (see Definition \ref{def-TC}). Let $\mcQ$ be a $\bC^*$-equivariant $\bQ$-line bundle over $\mcY$, and $\Psi_\mcQ=\{\psi_\mcQ(s)\}$ be a smooth $S^1$-invariant Hermitian metric on $\mcQ$. If $\{\phi_m\}\subset\mcH^\NA(L)$ is a sequence converging strongly to $\phi$ and $\Phi_m=\{\vphi_m(s)\}$ are the associated maximal geodesic rays, then we have
\begin{equation}\label{eq-EQslopeconv}
\lim_{s\rightarrow+\infty}\frac{\bfE^{\ddc\psi_\mcQ (s)}(\vphi(s))}{s}=\lim_{m\rightarrow+\infty}\lim_{s\rightarrow+\infty} \frac{\bfE^{\ddc\psi_\mcQ (s)}(\vphi_m(s))}{s}.
\end{equation}
Moreover we have the identity:
\begin{equation}\label{eq-f'inf}
\lim_{s\rightarrow+\infty} \frac{\bfE^{\ddc\psi_\mcQ (s)}(\vphi(s))}{s}=(\bfE^\cQ)^\NA(\phi),
\end{equation}

\end{thm}
\begin{proof}
Since the method of proof is the same as the proof of Theorem \ref{thm-Echislope}, we just sketch the proof. As in the previous proof, because \eqref{eq-f'inf} holds when $\phi\in \mcH^\NA$ by Proposition \ref{prop-Echislope2}, it suffices to prove the convergence \eqref{eq-EQslopeconv}
when $\{\phi_m\}$ is a decreasing sequence converging to $\phi$.

Using the functorial embedded resolution of singularities by a sequence of blowing-ups along smooth subvarieties, we can find a model $\pi_{\mcY'}: \mcY'\rightarrow \bC$ with $\bC^*$-equivariant birational morphisms $p: \mcY'\rightarrow X_\bC$ and $q: \mcY'\rightarrow \mcY$, and moreover there is a $\pi_{\mcY'}$-ample line bundle $\mcL'$ such that $(\mcY', \mcL')$ becomes an ample test configuration for $(X, L)$. There exists $\ell \gg 1$ such that  
$q^*\mcQ+\ell \mcL'$ is $\pi_{\mcY'}$-ample. As a consequence, there is a smooth Hermitian metric $\Phi_{\mcL'}$ such that both $\Phi_{\mcL'}=\{\vphi_{\mcL'}(s)\}$ and $\Psi_{\mcQ}+\ell \cdot \Phi_{\mcL'}$ have positive curvature forms. We just need to prove the convergence when $(\mcQ, \Psi_\mcQ)$ is replaced by the $\pi_{\mcY'}$-positive Hermitian $\bQ$-line bundles $$(\mcL', \Phi_{\mcL'}), \quad (q^*\mcQ+ \ell \mcL', \Psi_{\mcQ}+\ell \cdot \Phi_{\mcL'}).$$ Moreover when  $\ell \gg 1$, $2\ell \mcL'-(q^*\mcQ+\ell \mcL')$ is also $\pi_{\mcY'}$-ample. By using the assumption $\Phi_m\ge \Phi$ and co-cycle property as before, it suffices to prove the convergence for the $\pi_{\mcY'}$-ample Hermitian $\bQ$-line bundle $(\mcL', \Phi_{\mcL'})$. 
Now we can carry out exactly same arguments as the proof of Theorem \ref{thm-Echislope}, replacing $\psi(s)$ by $\psi_{\mcL'}(s)$. The only place we need to modify is replace \eqref{eq-Ilinear} by the estimate:
\begin{equation}
\bfI(\psi_{\mcL'}(s), \vphi(s))\le C(\bfI(\psi, \psi_{\mcL'}(s))+\bfI(\psi, \vphi(s)))\le Cs,
\end{equation} 
which follows from the asymptotic expansion of the $\bfI$ functional stated in Proposition \ref{prop-ARvsNA}.
\end{proof}



\section{Slope of entropy and non-Archimedean entropy }\label{sec-YTD1}

In order the compare the slope $\bfH$ along a maximal geodesic ray with $\bfH^\NA$, we first reformulate the $\bfH^\NA$ functional.

\begin{defn}\label{def-MNA}
Assume $(\mcX, \mcL)$ is a semi-ample test configuration and let $\phi_{(\mcX, \mcL)}\in \mcH^\NA$ be the associated non-Archimedean metric (see \eqref{eq-phiH}).
For any $\mcY\in \fMO$ (see Definition \ref{def-TC}), let $\mcZ$ be a common refinement (see Definition \ref{def-TC}) of $X\times\bC$, $\mcX$ and $\mcY$ with dominating morphism as shown in the following diagram:
\begin{equation}\label{eq-common}
\xymatrix{
& \ar_{p_1}[ld] \mcZ \ar^{p_0}[d]   \ar^{p_2}[rd] & \\
\mcX   & \ar@{-->}[l]  X_\bC  \ar@{-->}[r] &  \mcY.  
}
\end{equation}
Set:
\begin{eqnarray}
\bfH^\NA(\phi_{(\mcX, \mcL)}; \mcY)&:=&\bfH^\NA(\mcX, \mcL; \mcY):=\qV  (\bar{p}_2^*K^{\log}_{\bar{\mcY}/\bP^1}-\bar{p}_0^*K^{\log}_{X_{\bP^1}/\bP^1})\cdot (\bar{p}_1^*\bar{\mcL})^{\cdot n} \label{eq-HNAcY} \\
\bfS^\NA(\phi_{(\mcX, \mcL)}; \mcY)&=&\bfS^\NA(\mcX, \mcL; \mcY):=\qV  \bar{p}_2^*K^{\log}_{\bar{\mcY}/\bP^1}\cdot \bar{p}_1^*\bar{\mcL}^{\cdot n}.
\end{eqnarray}
More generally, for any $\phi\in \cE^{1,\NA}(L)$, let $\{\phi_m\}\subset \mcH^\NA(L)$ be a decreasing sequence converging to $\phi$ and define:
\begin{eqnarray}\label{eq-SNAphicY}
\bfH^\NA(\phi; \mcY)&:=&\qV K^{\log}_{\bar{\mcY}/X_{\bP^1}}\cdot \phi^{\cdot n}:=\lim_{m\rightarrow+\infty} \bfH^\NA(\phi_m; \mcY),\\
\bfS^\NA(\phi; \mcY)&:=&\qV K^{\log}_{\bar{\mcY}/\bP^1}\cdot \phi^{\cdot n}:=\lim_{m\rightarrow+\infty} \bfS^\NA(\phi_m; \mcY).
\end{eqnarray}

\end{defn}
Note that the convergence of limit follows from Proposition \ref{prop-strongcont}. 
Next we give a reformulation of $\bfH^\NA(\phi)$.
\begin{prop}\label{prop-intA}
For any $\phi\in \cE^{1,\NA}(L)$. 
we have the formula:
\begin{eqnarray}
\bfH^\NA(\phi)&=&\sup_{\mcY\in \fSN}\int_{X^\NA}A_X(r_{\mcY}(x))\MA^\NA(\phi)(x)\\
&=&\sup_{\mcY\in \fSN}\bfH^\NA(\phi; \mcY)=\sup_{\mcY\in \fSN} \qV  K^{\log}_{\bar{\mcY}/X_{\bP^1}}\cdot \phi^{\cdot n}. \label{eq-supHcY}
\end{eqnarray}

\end{prop}
\begin{proof}
The first identity follows from \eqref{eq-AXsup} and the monotone convergence theorem which is valid for increasing net and Radon measure (see \cite[7.12]{Fol99}). 

To prove \eqref{eq-supHcY}, we will prove the identity:
\begin{equation}\label{eq-HcY}
\bfH^\NA(\phi; \mcY)=\qV \int_{X^{\NA}}A_X(r_\mcY(v))\MA^\NA(\phi)(v).
\end{equation}
Let $\{\phi_m\}\subset\mcH^\NA(L)$ be a sequence decreasing to $\phi$.
For any $\mcY\in \fSN$ we have:
\begin{eqnarray*}
\bfH^\NA(\phi; \mcY)&=&\qV K^{\log}_{\bar{\mcY}/X_{\bP^1}}\cdot \phi^{\cdot n}=\lim_{m\rightarrow+\infty} \qV K^{\log}_{\bar{\mcY}/X_{\bP^1}}\cdot \phi_m^{\cdot n}.
\end{eqnarray*}
We claim that there is an identity:
\begin{eqnarray}\label{eq-intint}
\qV K^{\log}_{\bar{\mcY}/X_{\bP^1}}\cdot \phi_m^{\cdot n} &=&\qV \int_{X^\NA}A_X(r_\mcY (v))\MA^\NA(\phi_m)(v),
\end{eqnarray}
where $r_\mcY: X^\NA\rightarrow \Delta_{\mcY}$ is the retraction map (see \cite[4]{BoJ18a}). 

Assuming this claim, we let $m\rightarrow+\infty$ in \eqref{eq-intint} to get the identity \eqref{eq-HcY}. Indeed, because $v\mapsto A_X(r_\mcY(v))$ is continuous (see \cite[Proof of Lemma 5.7]{JM12}) and $\MA^\NA(\phi_m)$ converges to $\MA^\NA(\phi)$ weakly (\cite[Corollary 6.12]{BoJ18a}), the right-hand-side of \eqref{eq-intint} converges to the right-hand-side of \eqref{eq-HcY}.

Finally we verify the identity \eqref{eq-intint}. Note that $D:=K^{\log}_{\bar{\mcY}/X_{\bP^1}}$ corresponds to a function 
$f: X^\NA\rightarrow \bR$ whose value at any divisorial point $x\in X^{\rm div}_\bQ$ is given by:
\begin{eqnarray}\label{eq-fD}
f_D(x)&=&G(x)(D)=G(x)\left(\sum_i A_{X_{\bP^1}}(E_i) E_i+X'_0-\mcY_0\right)\nonumber \\
&=&A_{X_{\bP^1}}({\rm ev }_{\mcY}(G(x)))-1=A_X(r_{\mcY}(x)).
\end{eqnarray}
On the other hand, we have the identity:
\begin{eqnarray}\label{eq-intMA}
D\cdot \mcL^{\cdot n}&=& \sum_j \ord_{F_j}(D) F_j \cdot \mcL^{\cdot n}=\sum_j \ord_{F_j}(D) \left(\mcL|_{F_j}\right)^{\cdot n}\nonumber \\
&=&\sum_j b_j G(r(b_j^{-1}\ord_{F_j}))(D) \left(\mcL|_{F_j}\right)^{\cdot n}=\sum_j b_j f_D(x_{F_j}) \left(\mcL|_{F_j}\right)^{\cdot n}\nonumber \\
&=&\int_{X^\NA} f_D(x) \MA^\NA(\phi)(x).
\end{eqnarray}
Combining \eqref{eq-fD} and \eqref{eq-intMA}, we indeed get the identity \eqref{eq-intint}.
\end{proof}

Note that the identity $\bfH^\NA$ in \eqref{eq-supHcY} is a non-Archimedean analogue of the well-known supremum characterization of Archimedean entropy:
\begin{equation}
\bfH_\Omega(\vphi)=\sup\left\{\int_X \log \frac{\nu}{\Omega}(\ddc\vphi)^n; \nu \text{ a probability measure s.t.} \log\frac{\nu}{\Omega}\in C^0(X) \right\}.
\end{equation}
In particular, the following lower semi-continuity is always true: for any $\{\phi_m\}\subset \mcH^\NA$ converging strongly to $\phi\in \cE^{1,\NA}$, we have:
\begin{equation}
\bfH^\NA(\phi)\le \liminf_{m\rightarrow+\infty} \bfH^\NA(\phi_m).
\end{equation}
Now we can prove Theorem \ref{thm-grslope}.
\begin{thm}\label{thm-HNA}
\begin{enumerate}
\item
For any $\phi\in \mcE^{1,\NA}(L)$, let $\Phi=\{\vphi(s)\}$ be the associated maximal geodesic ray. We have the following inequalities:
\begin{equation}\label{eq-HslopeNA}
\bfH'^\infty(\Phi)\ge \bfH^\NA(\phi), \quad \bfM'^\infty(\Phi)\ge \bfM^\NA(\phi).
\end{equation}
\item
If $\phi=\phi_{(\mcX, \mcL)}\in \mcH^\NA(L)$, then we have:
\begin{eqnarray}
\bfH'^\infty(\Phi)&=&\bfH^\NA(\phi)=\bfH^\NA(\phi; \mcX) \label{eq-HslopeH} \\
\bfM'^\infty(\Phi)&=&\bfM^\NA(\phi).  \label{eq-MslopeH}
\end{eqnarray}
\end{enumerate}
\end{thm}

\begin{proof}
The second inequality in \eqref{eq-HslopeNA} follows from the first one and the convergence results in Theorem \ref{thm-maximal} and Theorem \ref{thm-Echislope}. So we just need to prove the inequality for the entropy part. 

Choose any $\mcY\in \fDS$ with projection $\pi: \mcY\rightarrow \bC$.
Let $\nu$ be a smooth Hermitian metric on $-K^{\log}_{\bar{\mcY}/\bP^1}$.  Set $X_t=\pi^{-1}(\{t\})$. Away from $\mcY_0$, $\nu|_{X_t}$ is a smooth volume form on $X_t$. The entropy part can be re-written as:
\begin{eqnarray}
\bfH_\Omega(\vphi)&=&\qV \int_X \log \frac{(\ddc\vphi)^n}{\Omega}(\ddc\vphi)^n\nonumber \\
&=&\qV \int_X\log \frac{(\ddc\vphi)^n}{\nu}(\ddc\vphi)^n+\int_X \log \frac{\nu}{\Omega}(\ddc\psi)^n\nonumber \\
&&\hskip 2cm +\qV \int_X\log\frac{\nu}{\Omega}\left[(\ddc\vphi)^n-(\ddc\psi)^n\right].
\end{eqnarray}
We deal with each term as follows. 
\begin{itemize}
\item With $s=-\log|t|^2$, by Jensen's inequality, 
\begin{equation}
\int_X \log\frac{(\ddc\vphi)^n}{\nu}(\ddc\vphi)^n\ge -V \cdot \log\int_{X_t}\frac{\nu}{V}.
\end{equation}
Moreover, by \cite[Lemma 3.11]{BHJ19} we know that $\int_{X}\nu|_{X_t}=O(s^d)$ where $d$ is the dimension of the dual complex of $\mcX_0$. 
\item We assume that $\mcY$ dominates $X_\bC$ via the birational morphism $\rho: \mcY\rightarrow X_\bC$. If $E_i$ denotes the exceptional divisor of $\rho$, then we have
\begin{equation}
K^{\log}_{\mcY/\bC}=\rho^*K^{\log}_{X_\bC/\bC}+\sum_i A_i E_i
\end{equation}
with $A_i\ge 0$ since $(X_\bC, X_0)$ is plt.
If $\sigma$ is a local generator of $K^{\log}_{X_\bC/\bC}$ which is nothing but a holomorphic $n$-form on $X$ pulled back to $X_\bC$, then, under a local holomorphic chart $\{z_i\}$, $\rho^*\sigma=\sigma' \cdot \prod_i z_i^{A_i}$ is a local section of $K^{\log}_{\mcY/\bC}$, where $\sigma'$ is a local generator of $K^{\log}_{\mcY/\bC}$ and $E_i=\{z_i=0\}$. So we get:
\begin{equation}
\Omega\sim |\rho^*\sigma|^2=|\sigma'|^2 \prod_{i}|z_i|^{2A_i}\le C \nu
\end{equation}
where $C$ is a constant independent of $s>0$.
So we get:
\begin{equation}
\int_X \log \frac{\nu}{\Omega}(\ddc\psi)^n\ge -V\cdot \log C 
\end{equation}
\item We use the identity \eqref{eq-E2chi} to get:
\begin{equation}
\int_X\log \frac{\nu}{\Omega}[(\ddc\vphi)^n-(\ddc\psi)^n]=\bfE^{-Ric(\nu)}_{\psi}(\vphi)-\bfE^{-Ric(\Omega)}_\psi(\vphi).
\end{equation}
\end{itemize}
In summary, we get the estimate:
\begin{equation}
\bfH_\Omega(\vphi(s))\ge -V \cdot \log (O(s^d))-V\cdot \log C+\bfE^{-Ric(\nu)}_{\psi}(\vphi)-\bfE^{-Ric(\Omega)}_\psi(\vphi). \label{eq-Jensen}
\end{equation}
Now take slopes on both sides of \eqref{eq-Jensen} and use Theorem \ref{thm-EcQslopeconv} together with \eqref{eq-intEcQNA} to get:
\begin{eqnarray*}
\bfH'^\infty(\Phi)&\ge& (\bfE^{K^{\log}_{\mcY/\bC}})^\NA(\phi)-(\bfE^{K_X})^\NA(\phi)\\
&=& (K^{\log}_{\bar{\mcY}/\bP^1}-K^{\log}_{X_{\bP^1}/\bP^1})\cdot (\phi^{\cdot n}-\phi_\triv^{\cdot n})\\
&=&K^{\log}_{\bar{\mcY}/X_{\bP^1}}\cdot \phi^{\cdot n}=\bfH^\NA(\phi; \mcY),
\end{eqnarray*}
where we used the vanishing identity:
$$K^{\log}_{\bar{\mcY}/X_{\bP^1}}\cdot \phi_{\triv}^{\cdot n}=K^{\log}_{\bar{\mcY}/X_{\bP^1}}\cdot \pi^*L^{\cdot n}=0.
$$ 
So, by using \eqref{eq-supHcY} we get the inequality \eqref{eq-HslopeNA}:
\begin{equation}
\bfH'^\infty(\Phi)\ge \sup_{\mcY\in \fSN}\bfH^\NA(\phi; \mcY)=\bfH^\NA(\phi).
\end{equation}


Finally, to prove \eqref{eq-MslopeH} and hence \eqref{eq-HslopeH}, we just need to show $\bfM'^\infty(\Phi)\le \bfM^\NA(\mcX, \mcL)$. But this has been proved in \cite[Proposition 5.1]{Xia19}. We sketch the proof there for the reader's convenience. Let $\tilde{\Phi}=\{\tilde{\vphi}(s)\}$ be a smooth $S^1$-invariant metric on $\mcL$ with $\tilde{\vphi}(0)=\psi$. Let $c_s(r), r\in [0, s]$ be the $C^{1,1}$ geodesic segment connecting $\psi$ and $\tilde{\vphi}(s)$.  By the convexity of Mabuchi energy:
\begin{equation}\label{eq-MslopeNA}
\bfM(c_s(r))\le \frac{r}{s}\bfM(\tilde{\vphi}(s)).
\end{equation} 
For any compact interval of $r$, $\{c_s(r)\}$ converges uniformly in strong topology to $\Phi=\{\vphi(r)\}$ which is the geodesic ray associated to $(\mcX, \mcL)$. Indeed by the convexity in Theorem \ref{thm-d1convex}, we have:
\begin{equation}
d_1(c_s(r), \vphi(r))\le \frac{r}{s} d_1(\tilde{\vphi}(s),\vphi(s))\le C \frac{r}{s}.
\end{equation}
The last inequality uses the fact that $|\Phi-\tilde{\Phi}|\le C$ and hence $d_1(\tilde{\vphi}(s), \vphi(s))\le C \bfI_1(\tilde{\vphi}(s), \vphi(s))$ is also uniformly bounded independent of $s$.
The claimed convergence then follows by letting $s\rightarrow+\infty$.
In the inequality \eqref{eq-MslopeNA}, by letting $s\rightarrow+\infty$ and using the lower semicontinuity of Mabuchi energy, we get:
\begin{equation}
\bfM(\vphi(r))\le r \bfM'^\infty(\tilde{\Phi})=r \bfM^\NA(\mcX, \mcL).
\end{equation}
The last identity used the slope formula from \cite{BHJ17}.
Dividing both sides by $r$ and letting $r\rightarrow+\infty$, we get the inequality.

\end{proof}


\begin{rem}
Here we give a different proof of the inequality $\bfM'^\infty(\Phi)\ge \bfM^\NA(\Phi_\NA)$ in \eqref{eq-HslopeNA} when $\Phi=\Phi_{(\mcX, \mcL)}$.
We first construct a (special) $S^1$-invariant smooth subgeodesic ray $\tilde{\Phi}$ as follows.
Note that by \cite[Theorem 1.4]{CC18b}, we are free to choose the initial point of $\Phi$ without changing $\bfM'^\infty(\Phi)$.
Since $\mcL$ is assumed to be ample, it is well-known that the test configuration $(\mcX, \mcL)$ is associated to a one-parameter $\bC^*$-action on $\bP^{N_{p}-1}$ where $N_{p}=\dim H^0(X, p L)$ for $p\gg 1$. In other words, we can assume that there is a $\bC^*$-equivariant embedding $\iota: \mcX\rightarrow \bP^{N_{p}-1}\times\bC$ such that $\mcX=\overline{\{(\sigma_\eta(t)(X),t); t\in \bC\}}\subset \bP^{N_{p}-1}\times\bC$ where $\eta\in \mathfrak{gl}(N_{p},\bC)$ generates a $\bC^*$-action and $\sigma_\eta(t)=\exp(-(\log t) \eta)\in GL(N_{p},\bC)$. Let $\psi_{\FS}$ be the standard Fubini-Study metric on the hyperplane bundle over $\bP^{N_{p}-1}$. Set $\tilde{\vphi}(t)=\sigma_\eta(t)^*\vphi^{1/p}_{\FS}|_{\sigma_\eta(t)(X)}$.

Now let $\Phi$ be the geodesic ray associated to $(\mcX, \mcL)$. It is known that $\Phi$ is $C^{1,1}$ on $X\times \bC^*$ (\cite{CTW18, PS10}). Moreover $\Phi$ and $\tilde{\Phi}$ is $L^\infty$ comparable: there exists $C>0$ such that $|\Phi-\tilde{\Phi}|\le C$ over $X\times(\bB\setminus\{0\})$. 
By the co-cycle property of $\bfM$ (see \eqref{eq-cocycle}), we have $\bfM_\psi(\vphi)=\bfM_\psi(\tilde{\vphi})+\bfM_{\tilde{\vphi}}(\vphi)$. So it is enough to show that $\bfM_{\tilde{\vphi}}(\vphi)$ is bounded from below. By substituting $\psi$ and $\Omega$ by $\tilde{\vphi}$ and $(\ddc\tilde{\vphi})^n$ respectively in \eqref{eq-Mabuchi}, we get:
\begin{eqnarray*}
\bfM_{\tilde{\vphi}}(\vphi)&=&\qV \int_X \log\frac{(\ddc\vphi)^n}{(\ddc\tilde{\vphi})^n}(\ddc\vphi)^n+ \bfE^{-Ric(\ddc\tilde{\vphi})}_{\tilde{\vphi}}(\vphi)+\frac{\udS}{n+1} \bfE_{\tilde{\vphi}}(\vphi).
\end{eqnarray*}
We estimate each term separately. By Jensen's formula applied to the probability measure $((2\pi)^nV)^{-1}(\ddc\vphi)^n$, the entropy part is nonnegative. 
The (negative) Ricci energy part can be rewritten as:
\begin{eqnarray*}
&& \bfE^{-Ric(\ddc\tilde{\vphi})}_{\tilde{\vphi}}(\vphi)\\
&=&\sum_{k=0}^{n-1}\int_X (\vphi-\tilde{\vphi})(-Ric(\ddc\tilde{\vphi}))\wedge (\ddc\vphi)^k\wedge (\ddc \tilde{\vphi})^{n-1-k}\\
&=&-\sum_{k=0}^{n-1}\int_X (\tilde{\vphi}-\vphi)\left[-Ric(\ddc\tilde{\vphi})+p (n+1)(\ddc\tilde{\vphi})\right]\wedge (\ddc\vphi)^k\wedge (\ddc\tilde{\vphi})^{n-1-k}\\
&&+p(n+1)\sum_{k=0}^{n-1}\int_X (\tilde{\vphi}-{\vphi}) (\ddc\vphi)^{k}\wedge (\ddc\tilde{\vphi})^{n-k}.
\end{eqnarray*}

Now the observation is that $\ddc\tilde{\vphi}$, which is the restriction of Fubini-Study metric, satisfies $Ric(\ddc\tilde{\vphi})\le p (n+1) \ddc\tilde{\vphi}$ (by Gauss-Codazzi equation). Because $|\vphi- \tilde{\vphi}|\le C$, we easily get that $\bfE^{-Ric(\ddc\tilde{\vphi})}_{\tilde{\vphi}}(\vphi)$ is uniformly bounded and $\bfE_{\tilde{\vphi}}(\vphi)$ is also easily bounded. 

\end{rem}

\section{Existence results for cscK metrics}\label{sec-YTD}

\subsection{Uniform stability for model filtrations}\label{sec-YTD2}

In the following discussion, we will use the notations from section \ref{sec-twist}.
A main goal in this section is to prove Theorem \ref{thm-existfiltr}. We will first prove a weaker statement which says that uniform K-stability over $\cE^{1,\NA}$ implies the cscK (Proposition \ref{prop-exist1}).
This is indeed straightforward given the results obtained so far. To improve this result, we will resort to the works on non-Archimedean Monge-Amp\`{e}re equations by Boucksom-Favre-Jonsson and Boucksom-Jonsson.

We also note that the proof is along the similar line as the proof in \cite{Li19} of $\bG$-uniform version of Yau-Tian-Donaldson conjecture for all (possibly singular) Fano varieties.
The proof in \cite{Li19} depends on earlier works of Berman-Boucksom-Jonsson (\cite{BBJ18}), Hisamoto (\cite{His16b}) and our perturbative approach in \cite{LTW19}, and uses a new valuative criterion for $\bG$-uniform stability. 
The following proof will also slightly streamline the argument in \cite{Li19} in the smooth case and shows that, when $X$ is smooth Fano, one could avoid the use of valuative criterion by a direct estimate of slope of $\bfJ$ (see \eqref{eq-Jslopeuni}).

We first highlight the identities/inequalities needed  in the proof of the following existence results. They are
corollaries of Theorem \ref{thm-maximal}, Theorem \ref{thm-Echislope} and Theorem \ref{thm-grslope}:
\begin{cor}\label{thm-maxYTD}
Let $\Phi$ be a geodesic ray of finite Mabuchi slope. Then we have the identities:
\begin{equation}\label{eq-energyslope}
\bfE'^\infty(\Phi)=\bfE^\NA(\Phi_\NA), \quad \bfR'^\infty(\Phi)=\bfR^\NA(\Phi_\NA), \quad \inf_{\xi\in N_\bR}\bfJ'^\infty(\Phi_\xi)=\bfJ^\NA_\bT(\Phi_\NA),
\end{equation}
and the inequality:
\begin{equation}\label{eq-entropyslope}
\bfM'^\infty(\Phi)\ge \bfM^\NA(\Phi_\NA).
\end{equation}

\end{cor}

\begin{prop}\label{prop-exist1}
If $(X, L)$ is $\bG$-uniformly K-stable over $(\mcE^{1,\NA})^\bK$ (see Definition \ref{def-Kst}), then $(X, L)$ admits a cscK metric.
\end{prop}

\begin{proof}
By the previous works \cite{BDL17, BDL18, CC18a, CC18b, DR17} on analytic criterions for the existence of cscK, we just need to show that $\bfM$ is $\bG$-coercive, which means that there exist $\gamma>0$ and $C>0$ such that for any $\vphi\in \mcH(L)^{\bK}$, 
\begin{equation}
\bfM(\vphi)\ge \gamma\cdot \inf_{\sigma\in \bT}\bfJ(\sigma^*\vphi)-C.
\end{equation}
(see the beginning of this subsection for some notations.)

Assume that this is not true. 
There exists $\gamma_j\rightarrow 0^+$ and $\vphi_j\in \mcH(L)^\bK$ such that:
\begin{equation}
\bfM(\vphi_j)\le \gamma_j \bfJ(\vphi_j)-j, \quad \sup(\vphi_j-\psi)=0, 
\end{equation}
and (see \cite[Lemma 1.9]{His16b})
\begin{equation}\label{eq-vphijmin}
\bfJ(\vphi_j)=\inf_{\sigma\in \bT}\bfJ(\sigma^*\vphi_j). 
\end{equation}

We argue as in \cite{BBJ18} and \cite{DH17}.
Connect $\psi$ and $\vphi_j$ by a geodesic ray $\Phi_j=\{\vphi_j(s); 0\le s\le S_j\}$ with $S_j=-\bfE(\vphi_j)$. Then
by Lemma \ref{lem-cJgrow} we have the inequality $\bfM\ge -C-\delta \bfJ$ which gives us:
\begin{equation}\label{eq-Jdiverge}
\bfJ(\vphi_j)=\bfJ(\vphi_j(S_j))\ge \frac{j-C}{\delta+\gamma_j}\rightarrow+\infty.
\end{equation}
We then know that 
\begin{equation}
-\bfE(\vphi_j)=\bfJ(\vphi_j)+O(1)=S_j\rightarrow+\infty.
\end{equation}

Using the convexity of $\bfM$ we get:
\begin{equation}
\bfM(\vphi_j(s))\le \frac{s}{S_j}\bfM(\vphi_j)\le \gamma_j s. 
\end{equation}
and by using \eqref{eq-cJgrow} we get
\begin{equation}
\bfH(\vphi_j(s))=\bfM(\vphi_j(s))-\cJ(\vphi_j(s))\le \gamma_j s+C+\delta \bfJ(\vphi_j(s))\le Cs.
\end{equation}

So by the compactness result (\cite[Theorem 2.17]{BBEGZ}) in $\cE^1$, $\Phi_j=\{\vphi_j(s)\}$ converges locally uniformly to a finite energy geodesic ray $\Phi=\{\vphi(s)\}\subset (\cE^1)^\bK$ satisfying:
\begin{enumerate}
\item The Mabuchi energy is decreasing along $\Phi$:
\begin{equation}\label{eq-Mabdec}
\bfM'^\infty(\Phi)\le 0.
\end{equation}

\item We have a normalization:
\begin{equation}
\bfE(\vphi(s))=-s, \quad \sup (\vphi(s)-\psi)=0.
\end{equation}
\end{enumerate}

Moreover we claim that the following inequality holds true:
\begin{equation}\label{eq-Jslopeuni}
\inf_{\xi\in N_\bR}\bfJ'^\infty(\Phi_\xi)= 1.
\end{equation}
Assuming this claim, we can prove Proposition \ref{prop-exist1}. Indeed we then have the inequality that contradicts \eqref{eq-Mabdec}.
\begin{equation}
\bfM'^\infty(\Phi)\ge \bfM^\NA(\phi)\ge \gamma \inf_{\xi\in N_\bR} \bfJ^\NA(\phi_\xi)=\gamma \inf_{\xi \in N_\bR} \bfJ'^\infty(\Phi_\xi)\ge \gamma>0.
\end{equation}
Here the first inequality follows from Proposition \eqref{thm-HNA}. The first identity is the assumption of stability over $\mcE^{1,\NA}$ and the second identity follows from the maximality of $\Phi$.

To verify the claim \eqref{eq-Jslopeuni}, we use the fact that there exists a universal constant $\mathfrak{C}=\mathfrak{C}(\psi)$ such that for any $\vphi\in \cE^1$:
\begin{equation}
\sup(\vphi-\psi)-\bfE(\vphi) \ge \bfJ(\vphi)=\Lam(\vphi)-\bfE(\vphi) \ge \sup(\vphi-\psi)-\bfE(\vphi)-\mathfrak{C}.
\end{equation}
In our case, $\sup(\vphi(s)-\psi)-\bfE(\vphi(s))=-\bfE(\vphi(s))=s$ is linear which implies $\bfJ'^\infty(\Phi)=1$.
Moreover for any $s_1, s_2\in \bR_{>0}$, we have:
\begin{equation}\label{eq-Jlinear}
\bfJ(\vphi(s_1))\ge s_1-\mathfrak{C}=\frac{s_1}{s_2}(s_2)-\mathfrak{C}\ge \frac{s_1}{s_2}\bfJ(\vphi(s_2))-\mathfrak{C}.
\end{equation}
Now we apply this inequality to $(\Phi_j)_\xi=\{\sigma_\xi(s)^*\vphi_j(s)\}$ for any $\xi\in N_\bR$ to get:
\begin{eqnarray*}
\bfJ(\sigma_\xi(s)^*\vphi_j(s))&\ge& \frac{s}{S_j}\bfJ(\sigma_\xi(S_j)^*\vphi_j(S_j))-\mathfrak{C}\\
&\ge& \frac{s}{S_j}\bfJ(\vphi_j(S_j))-\mathfrak{C} \quad\quad\quad  (\text{by } \eqref{eq-vphijmin})\\
&=& \frac{s}{S_j}(S_j+O(1))-\mathfrak{C}= s-\mathfrak{C}+\frac{s}{S_j}O(1) .
\end{eqnarray*}
Because, for a fixed $s$, as $j\rightarrow+\infty$, $\vphi_j(s)$ converges strongly to $\vphi(s)$, it is easy to see that $\sigma_\xi(s)^*\vphi_j(s)$ also converges strongly to $\sigma_\xi(s)^*\vphi(s)$. So for fixed $s$, letting $j\rightarrow+\infty$, we get:
\begin{equation}
\bfJ(\sigma_\xi(s)^*\vphi(s))\ge s-\mathfrak{C}.
\end{equation}
Dividing $s$ on both sides and letting $s\rightarrow+\infty$, we indeed get \eqref{eq-Jslopeuni}.

\end{proof}

To improve our result, we will show the following result which depends on the fundamental works on non-Archimedean Monge-Amp\`{e}re equations in \cite{BFJ15, BoJ18a}.
We point out that the technical inputs for this result all come from \cite{BFJ15, BoJ18a} and the argument here shows some of their powers.
\begin{prop}\label{prop-regularize}
For any $\phi\in (\cE^{1,\NA})^\bK$, there exists a sequence $\{\phi_j\}\subset (\PSH^{\fM,\NA})^\bK$ such that $\phi_j$ converges strongly to $\phi$ and $\bfH^\NA(\phi_j)\rightarrow \bfH^\NA(\phi)$.
\end{prop}
\begin{proof}
We will prove this in two steps. For simplicity of notations, we assume that $\bG=\{e\}$, since the following argument can be easily carried out in the $\bG$-equivariant manner.

\noindent
{\bf Step 1:} We first show that there exists a sequence $\{\phi_j\}_j\subset \PSH^{0, \NA}$ s.t. $\bfH^\NA(\phi_j)\rightarrow \bfH^\NA(\phi)$ and moreover for each $j$,  $\MA^\NA(\phi_j)$ is supported on the dual complex of some SNC model.

We set $\nu=\MA^\NA(\phi)$ and use the regularization process as in \cite{BoJ18a}.
For any SNC model $\mcX$ we set $\nu_\mcX=(r_{\mcX})_*\nu$. 
By \cite[Corollary 7.20]{BoJ18a} we have:
\begin{equation}
\bfE^{*\NA}(\nu)=\sup_{\mcX\in \fDS} \bfE^{*\NA}(\nu_{\mcX}).
\end{equation}
Choose a sequence $\mcX'_j$ such that $\nu'_j:=\nu_{\mcX'_j}$ satisfies $\bfE^{*\NA}(\nu'_j)\rightarrow \bfE^{*\NA}(\nu)$. 

On the other hand, by \cite[Lemma 2.3]{BoJ18b}, we have the identity:
\begin{equation}
\bfH^\NA(\phi)=\int_{X^\NA} A_X(x)\MA(\phi)(x)=\sup_{\mcX\in \fSN}\int_{X^\NA}A_X \nu_{\mcX}.
\end{equation}
So we can choose a sequence of SNC models $\mcX''_j$ such that $\nu''_j:=(r_{\mcX''_j})_*\nu$ satisfies:
\begin{equation}
\bfH^\NA(\phi)=\lim_{j\rightarrow+\infty} \int_{X^\NA}A_X \nu''_j.
\end{equation}
Let $\mcX_j$ be an SNC model satisfying $\mcX_j\ge \mcX'_j$ and $\mcX_j\ge \mcX''_j$ and set $\nu_j=(r_{\mcX_j})_*\nu$. Recall that $(A_X\circ r_{\mcX})_{\mcX\in \fDS}$ is increasing, and $((\phi-\phi_\triv)\circ r_{\mcX})_{\mcX\in \fDS})$ is decreasing (see Theorem \ref{thm-phidec}) which by formula \eqref{eq-E*NA} implies that $\bfE^{*\NA}(\nu_j)\ge \bfE^{*\NA}(\nu'_j)$. So we easily get:
\begin{equation}
\bfE^{*\NA}(\nu)=\lim_{j\rightarrow+\infty} \bfE^{*\NA}(\nu_j), \quad  \bfH^\NA(\phi)=\lim_{j\rightarrow+\infty} \int_{X^\NA}A_X \nu_j.
\end{equation}
Now set $\phi_j=(\MA^\NA)^{-1}(\nu_j)$ (see Theorem \ref{thm-NACY}).  By \cite[Lemma 7.23]{BoJ18a}, we have $\phi_j\rightarrow \phi$ in the strong topology. 
Because $\nu_j$ is supported on a dual complex, by \cite[Theorem A]{BFJ15}, $\phi_j$ is a continuous metric.

\noindent
{\bf Step 2:}
The above step reduces the problem to the situation where $\nu$ is a Radon measure supported on a dual complex of a SNC model $\mcX$. Fix such a measure, set $\phi=\MA^{-1}(\nu)\in \PSH^{0, \NA}$. By Multiplier Approximation, there exists a decreasing sequence $\{\phi_m\}_m\subset \mcH^\NA$ converging to $\phi$. In particular $\phi_m$ converges to $\phi$ strongly. Actually by \cite[8.3]{BFJ15}, $\phi_m$ converges to $\phi$ uniformly. As a consequence $\nu_m:=\MA^\NA(\phi_m)$ converges to $\MA^\NA(\phi)$ strongly, which means that (see Definition \ref{def-M1NA}):
\begin{equation}\label{eq-E*conv1}
\nu_m\overset{w}{\longrightarrow} \nu, \quad \bfE^{*\NA}(\nu_m)\rightarrow \bfE^{*\NA}(\nu).
\end{equation}

Now set $\nu'_m=(r_{\mcX})_*\nu_m$. Then the measure $\nu'_m$ is supported on finitely many divisorial points $\{x^{(m)}_1,\dots, x^{(m)}_{p_m}\}$ contained in $\Delta_\mcX$.
Moreover we claim that  
$\nu'_m$ converges to $\nu$ strongly. 

Assuming this claim, we can set $\phi'_m=(\MA^\NA)^{-1}(\nu'_m)$. Then by \cite[Lemma 8.5, Proposition 8.6]{BFJ15}, $\phi'_m=\phi_\triv+P(f_{\mcL'_m})$ (see \eqref{eq-upenv}) where $f_{\mcL'_m}$ is a model function on a SNC model $\mcX'_m$ and $\mcL'_m=\rho'^*_m L+D'_m$ is $D'_m=\sum_{k=1}^{p_m} t_i E_i$ an $\bR$-divisor supported on the divisors $E_i$ that corresponds to the divisorial point $x^{(m)}_i$. By perturbing and decreasing the coefficients of $E_i$, we can assume that $f_{\mcL'_m}$ has rational values at vertices. Note that this perturbation will also perturb the Monge-Amp\`{e}re measure in the strong topology (see \cite[Lemma 5.24]{BoJ18a}) and does not change the property that $\MA^\NA(\phi'_m)$ is supported on the fixed dual complex $\Delta_\mcX$ (see \cite[Lemma 8.5]{BFJ15} or \cite{Li20b}). So we can assume $\phi'_m$ is a sequence of envelopes of rational model functions 
and satisfies  $\phi'_m\rightarrow \phi$ strongly. Because the log discrepancy function $A_X$ is continuous on $\Delta_\mcX$, we get the convergence:
\begin{equation}
\int_{X^\NA} A_X(x)\MA^\NA(\phi'_m)(x)=\int_{\Delta_{\mcX}}A_X \nu'_m\overset{m\rightarrow+\infty}{\longrightarrow} \int_{\Delta_{\mcX}}A_X \nu=\int_{X^\NA}A_X(x)\MA^\NA(\phi)(x).
\end{equation}
Now we verify the claim. First we verify that $\nu'_m=(r_\mcX)_*\nu_m$ converges to $\nu$ weakly.  
Indeed, for any continuous function $u\in C^0(X^\NA)$, we have:
\begin{equation}
\int_{X^\NA} u \nu'_m=\int_{X^\NA} u\circ r_{\mcX} \nu_m\longrightarrow \int_{X^\NA} u\circ r_\mcX \nu=\int_{X^\NA} u \nu,
\end{equation}
where we used $(r_\mcX)_*\nu=\nu$.
Finally we show that $\bfE^{*\NA}(\nu'_m)\rightarrow \bfE^{*\NA}(\nu)$. To see this, note that $\bfE^{*\NA}(\nu'_m)=\bfE^{*\NA}((r_\mcX)_*\nu_m)\le \bfE^{*\NA}(\nu_m)$, because $\phi-\phi_\triv\le (\phi-\phi_\triv)\circ r_{\mcX}$ (see Theorem \ref{thm-phidec}). So by \eqref{eq-E*conv1}, we have:
\begin{equation}
\limsup_{m\rightarrow+\infty} \bfE^{*\NA}(\nu'_m)\le \limsup_{m\rightarrow+\infty}\bfE^{*\NA}(\nu_m)=\bfE^{*\NA}(\nu).
\end{equation}

On the other hand, by using the defining formula \eqref{eq-E*NA} of $\bfE^{*\NA}(\nu'_m)$
\begin{eqnarray*}
\bfE^{*\NA}(\nu'_m)&=&\sup_{\tilde{\phi} \in \cE^{1,\NA}}\left\{\bfE^\NA(\tilde{\phi})-\int_{X^\NA}(\tilde{\phi}-\phi_\triv)\nu'_m\right\}\\
&=&\sup_{\tilde{\phi}\in \cE^{1,\NA}}\left\{\bfE^\NA(\tilde{\phi})-\int_{X^\NA}(\tilde{\phi}-\phi_\triv)\circ r_\mcX \nu_m\right\},
\end{eqnarray*}
we use the continuity of $(\tilde{\phi}-\phi_\triv)\circ r_\mcX$ and the convergence of $\nu_m$ to $\nu$ to easily get that:
\begin{eqnarray*}
\liminf_{m\rightarrow+\infty} \bfE^{*\NA}(\nu'_m)&\ge& \sup_{\tilde{\phi}\in \mcE^{1,\NA}}\left\{\bfE^\NA(\tilde{\phi})-\int_{X^\NA}(\tilde{\phi}-\phi_\triv)\circ r_\mcX \nu\right\}\\
&=&\sup_{\tilde{\phi}\in \cE^{1,\NA}}\left\{\bfE^\NA(\tilde{\phi})-\int_{X^\NA}(\tilde{\phi}-\phi_\triv)\nu \right\}\\
&=&\bfE^{*\NA}(\nu).
\end{eqnarray*}
So we get the conclusion.

\end{proof}

We now state a useful lemma.
\begin{lem}\label{lem-JTNAconv}
If $\phi_j\rightarrow \phi$ in the strong topology, then $\bfJ^\NA_\bT(\phi_j)\rightarrow \bfJ^\NA_\bT(\phi)$.
\end{lem}
\begin{proof}
Because $\bfE^\NA(\phi_{j,\xi})-\bfE^\NA(\phi_\xi)=\bfE^\NA(\phi_j)-\bfE^\NA(\phi)$ for any $\xi\in N_\bR$, we know that $\phi_{j,\xi}\rightarrow \phi_\xi$ strongly. Note that
\begin{equation}
\bfJ^\NA_\bT(\phi_j)=\inf_{\xi\in N_\bR} (\Lam^\NA(\phi_{j,\xi})-\chw(\xi))-\bfE(\phi_j).
\end{equation}
From this it is easy to see that $\limsup_{j\rightarrow+\infty} \bfJ^\NA_\bT(\phi_j)\le \bfJ^\NA_\bT(\phi)$. To prove the other direction of inequality, we first claim that
there exists $C>0$ such that:
\begin{equation}\label{eq-xibd}
\bfJ^\NA_\bT(\phi_j)=\inf_{|\xi|\le C}\bfJ^\NA(\phi_{j,\xi}), \quad \bfJ^\NA_\bT(\phi)=\inf_{|\xi|\le C}\bfJ^\NA(\phi_\xi).
\end{equation}
 To see this, we first use an idea of Hisamoto (see \cite{His16b, His19}) which uses the quasi-triangle inequality (see \cite[Theorem 1.8]{BBEGZ} or \cite[Lemma 3.16]{BoJ18a}) to estimate
 for every $\xi\in N_\bR$:
\begin{eqnarray*}
\bfI^\NA(\phi_\triv, \phi_{\triv, \xi})&\le& C_n \max\{\bfI^\NA(\phi_\triv, \phi_{j,\xi}), \bfI^\NA(\phi_{j,\xi}, \phi_{\triv, \xi})\}.
\end{eqnarray*}
Here we used the identity \eqref{eq-Ixisame}. This easily implies that bounded $\bfJ(\phi_{j,\xi})$ implies bounded $\xi$, which implies \eqref{eq-xibd}.

Now we use the estimate from \cite[Corollary 3.21]{BoJ18a}:
\begin{eqnarray*}
\left|\bfJ^\NA(\phi_{j,\xi})-\bfJ^\NA(\phi_\xi)\right|\le C_n \bfI^\NA(\phi_{j,\xi}, \phi_\xi) \max\left(\bfJ^\NA(\phi_\xi), \bfJ^\NA(\phi_{j,\xi})\right)^{1-2^{-n}}
\end{eqnarray*}
So the functions $\xi\mapsto \bfJ^\NA(\phi_{j,\xi})$ converge to $\xi\mapsto \bfJ^\NA(\phi_\xi)$ uniformly over $\{|\xi|\le C\}$. From this we easily get the convergence.
\end{proof}

Combining the above discussion, we get the following existence result which says it is enough to check uniform K-stability for a special class of filtrations to get cscK:
\begin{thm}[=Theorem \ref{thm-existmodel}]\label{thm-existfiltr}
If $(X, L)$ is $\bG$-uniformly $K$-stable for model filtrations (see Definition \ref{def-Kst}), then the Mabuchi functional of $(X, L)$ is proper and hence there is a cscK metric. 
\end{thm}

\begin{proof}



Let $\Phi$ be the destabilizing geodesic ray from the proof of Theorem \ref{prop-exist1}. In particular, $\Phi$ satisfies 
\begin{equation}
\bfM'^\infty(\Phi)\le 0, \quad \inf_{\xi\in N_\bR} \bfJ'^\infty(\Phi_\xi)=1.
\end{equation}
Let $\phi=\Phi_\NA\in (\cE^{1,\NA})^\bK$ denote the non-Archimedean metric associated to $\Phi$ and let $\nu=\MA^\NA(\phi)$ be its Monge-Amp\`{e}re measure.  
By Proposition \ref{prop-regularize}, we can find $\phi_j\in (\PSH^{\fM,\NA})^\bK$ such that $\phi_j\rightarrow \phi$ strongly and $\bfH^\NA(\phi_j)\rightarrow \bfH^\NA(\phi)$. 
By Proposition \ref{prop-strongcont}, we also get $\bfM^\NA(\phi_j)\rightarrow \bfM^\NA(\phi)$. Now assume that $(X, L)$ is $\bG$-uniformly K-stable for model filtrations. Then we get:
\begin{eqnarray*}
\bfM'^\infty(\Phi)&\ge& \bfM^\NA(\phi)\quad \quad (\text{\eqref{eq-entropyslope} or Theorem \ref{thm-grslope}})\\
&=&\lim_{j\rightarrow+\infty} \bfM^\NA(\phi_j) \\
&\ge& \lim_{j\rightarrow+\infty}  \gamma \cdot \bfJ^\NA_\bT(\phi_j) \quad (\text{$\bG$-uniform stability })\\
&=&\gamma \cdot \bfJ^\NA_\bT(\phi)\quad (\text{Lemma } \ref{lem-JTNAconv}) \\
&=&\gamma \cdot \inf_{\xi\in N_\bR}\bfJ'^\infty(\Phi_\xi)=\gamma>0.  \quad (\text{\eqref{eq-energyslope} and \eqref{eq-Jslopeuni}})
\end{eqnarray*}
But this contradicts $\bfM'^\infty(\Phi)\le 0$.
\end{proof}
For comparison, we state and sketch the proof of a result that is known by the works \cite{His16b, Li19, Sze15} (see also \cite{BoJ18b}).
\begin{lem}
\label{lem-HisSze}
With the same notations as before, assume that $\bG$ contains a maximal torus of $\Aut(X, L)_0$. 
If $(X, L)$ admits a cscK metric, then for any filtration $\mcF=\mcF R_\bullet$ we have (with $\phi_m=\FS(\mcF R_m)$)
\begin{equation}\label{eq-Futproper}
\liminf_{m\rightarrow+\infty} \Fut(\phi_m) \ge \gamma\cdot \bfJ^\NA_\bT (\phi_\mcF).
\end{equation}
\end{lem}
\begin{proof}
Applying Darvas-Rubinstein's principle \cite{DR17} as in the way as \cite[Theorem 3.3]{His16b} and \cite[Theorem 2.15]{Li19}, we know that Mabuchi energy is $\bG$-coercive. By Hisamoto's slope formula \cite{His16b}, we then know that there exists $\gamma>0$ such that 
\begin{equation}
\bfM^\NA(\phi_m)\ge \gamma \cdot  \bfJ^\NA_\bT(\phi_m).
\end{equation}
Using base change, we get:
\begin{eqnarray*}
\Fut(\phi_m)\ge \gamma \cdot \bfJ^\NA_\bT(\phi_m).
\end{eqnarray*}
The result now follows by letting $m\rightarrow+\infty$ and Lemma \ref{lem-JTNAconv}.
\end{proof}
In \cite{Sze15}, the left-hand-side of \eqref{eq-Futproper} was defined to be $\Fut(\mcF)$. In general, by the lower semi-continuity of non-Archimedean entropy, we have the inequality
\begin{equation}
\bfM^\NA(\phi_\mcF)\le \liminf_{m\rightarrow+\infty} \bfM^\NA(\phi_m)\le \liminf_{m\rightarrow+\infty} \Fut(\phi_m)=\Fut(\mcF).
\end{equation}
Note that the second inequality is in general strict even for fixed metric in $\mcH^\NA$ if the central fibre of a corresponding test configuration has non-reduced components.
On the other hand,  the Conjecture \ref{conj-reg} implies that $\bfM^\NA(\phi_\mcF)=\lim_{m\rightarrow+\infty} \Fut(\tilde{\phi}_m)$ with some sequence $\tilde{\phi}_m\in \mcH^\NA(L)$.


\begin{prop}\label{prop-YTDconj}
If the Conjecture \ref{conj-reg} is true for $\phi\in \PSH^{\fM,\NA}(L)$ (see Definition \ref{def-modfil}), then $(X, L)$ admits a cscK metric if (and only if) $(X, L)$ is $\Aut(X, L)_0$-uniformly K-stable.
\end{prop}
\begin{proof}
Let $\Phi$ be the distabilizing geodesic ray satisfying:
\begin{equation}
\bfM'^\infty(\Phi)\le 0, \quad \inf_{\xi\in N_\bR}\bfJ'^\infty(\Phi_\xi)=1.
\end{equation} 
Set $\phi=\Phi_\NA$. Then by the proof Proposition \ref{prop-regularize}, there exists a sequence of $\phi_j\in \PSH^{\fM,\NA}$ such that 
\begin{equation}
\bfM'^\infty(\Phi)\ge \bfM^\NA(\phi)= \lim_{j\rightarrow+\infty} \bfM^\NA(\phi_j).
\end{equation}
Now assume that Conjecture \ref{conj-reg} is true for $\PSH^{\fM,\NA}(L)$. Then for each fixed $j$, there exists a sequence $\{\phi_{j,m}\}_m \subset \mcH^\NA(L)$ that converges strongly to $\phi_j$ and satisfies:
\begin{equation}
\lim_{m\rightarrow+\infty} \bfM^\NA(\phi_{j,m})=\bfM^\NA(\phi_j).
\end{equation} 
By using the $\bG$-uniform K-stability, there exists $\gamma>0$ such that 
\begin{eqnarray}\label{eq-dest2}
\bfM^\NA(\phi_{j,m})&\ge&
\gamma\cdot \bfJ^\NA_\bT(\phi_{j,m}).
\end{eqnarray}
Letting $m\rightarrow+\infty$, we get:
\begin{eqnarray}
\bfM^\NA(\phi_j)\ge \gamma\cdot \bfJ^\NA_\bT(\phi_j).
\end{eqnarray}

Letting $j\rightarrow+\infty$ and using \eqref{eq-entropyslope} and \eqref{eq-Jslopeuni}, we get:
\begin{eqnarray*}
\bfM'^\infty(\Phi)&\ge& \bfM^\NA(\phi)\ge \gamma\; \bfJ^\NA_\bT(\phi)=\gamma\inf_{\xi\in N_\bR} \bfJ'^\infty(\Phi_\xi) \ge \gamma. 
\end{eqnarray*}
But this contradicts \eqref{eq-Mabdec}.
\end{proof}
\begin{proof}[Proof of Theorem \ref{thm-toric}]
When $(X, L)$ is toric, set $\bG=\bT=(\bC^*)^r=((S^1)^r)^\bC=\bK^\bC$. 

Let $\Delta$ be the moment polytope and set $g_\Delta=\sup_{u\in \Delta}u$ defined on $N_\bR$. Then  by \cite[section 4]{BPS14}, the metrics in $(\PSH)^\bK\cap C^0(X^\NA)$ are in one-to-one correspondence to the continuous and convex functions $\phi$ on $N_\bR$ such that $\phi-g_\Delta$ extends to a continuous function on a compactification $\overline{N_\bR}$ that  is homeomorphic to the moment polytope $\Delta$. 
In this correspondence, the envelopes of psh metrics correspond to the usual upper envelopes of convex functions.

Because the upper envelope of family of piecewise $\bQ$-affine (meaning that the coefficients defining the affine functions are rational numbers) functions is still piecewise $\bQ$-affine, by the same argument \cite[Proof of Proposition 9.2]{BFJ15}, we know that $\cF\in (\PSH^{\fM,\NA})^\bK=(\mcH^\NA)^\bK$. The conclusion follows easily by using the above proposition. 
\end{proof}

\begin{rem}\label{rem-spherical}
It is well known that the above argument amounts to the algebraic fact that toric divisors on toric varieties always admit a rational Zariski decomposition (see \cite{BFJ09}). It has been pointed out to me by Y. Odaka that similar result as in Theorem \ref{thm-toric} hold for all smooth spherical varieties by using the fact that the total spaces of test configurations of spherical varieties are Mori Dream Spaces which indeed implies that all divisors admit rational Zariski decompositions.
\end{rem}

We end this section by showing the following relation between Conjectures in the introduction, which is suggested by S. Boucksom. 
\begin{lem}\label{lem-conjrel}
Conjecture \ref{conj-reg} implies Conjecture \ref{conj-Hslope} and Conjecture \ref{conj-entropy}.
\end{lem}
\begin{proof}
If $\phi\in \mcH^\NA$ and $\Phi$ is the associated geodesic ray, then by Theorem \ref{thm-grslope}.2 we have $\bfH'^\infty(\Phi)=\bfH^\NA(\phi)$. For a general $\phi\in \cE^{1,\NA}$, choose any sequence $\{\phi_m\}\subset \mcH^\NA$ that converges strongly to $\phi$. Let $\Phi_m$ be the associated geodesic ray. Then by the lower semicontinuity and convexity of Mabuchi energy, we have the inequality:
\begin{equation}
\frac{\bfM(\vphi(s))}{s}\le \liminf_{m\rightarrow+\infty} \frac{\bfM(\vphi_m(s))}{s}\le \liminf_{m\rightarrow+\infty} \bfM'^\infty(\Phi_m).
\end{equation} 
Letting $s\rightarrow+\infty$ and using Theorem \ref{thm-grslope}.1, we get the inequality:
\begin{equation}\label{eq-3Mslope}
\bfM^\NA(\phi)\le \bfM'^\infty(\Phi) \le \liminf_{m\rightarrow+\infty} \bfM^\NA(\phi_m).
\end{equation}
If Conjecture \ref{conj-reg} is true, then we can find a sequence $\{\phi_m\}$ satisfying $\bfM^\NA(\phi)\ge \lim_{m\rightarrow+\infty} \bfM^\NA(\phi_m)$. So for such a sequence, we indeed know that both inequalities in \eqref{eq-3Mslope} are identities.
\end{proof}

\subsection{$\cJ^\NA$-stability implies cscK metrics}\label{sec-Jst}

In the above section, we use the following estimate for the energy part of the Mabuchi functional.
\begin{lem}\label{lem-cJgrow}
There exists a constant $\delta>0$ and $C>0$ such that for any $\vphi\in \mcH$, we have:
\begin{equation}\label{eq-cJgrow}
-\delta \bfJ(\vphi)-C \le \cJ(\vphi) \le \delta \bfJ(\vphi)+C.
\end{equation}
\end{lem}
Note that the above estimates hold true for any $\vphi\in \cE^1(L)$ because $\bfJ$ and $\cJ$ are continuous under strong convergence of potentials (see \cite{BBEGZ}).
\begin{proof}
Because $\bfJ$ and $\bfI-\bfJ$ is comparable, in the above inequality we can use the latter to compare. Then it is easy to check that, for any $a\in \bR$,
\begin{eqnarray*}
\bfI-\bfJ&=&\bfI_\psi(\vphi)-\bfJ_\psi(\vphi)=\bfE_{\psi}^{\ddc \psi}(\vphi)-n \bfE_\psi(\vphi)\\
\mcJ+a \bf(I-\bfJ)&=&\bfE^{-Ric(\Omega)+a \omega}(\vphi)+(\ud{S}-n a)\bfE(\vphi).
\end{eqnarray*}
For any closed $(1,1)$-form, it is convenient to introduce:
\begin{equation}
\cJ^\chi(\vphi)=\bfE^{\chi}(\vphi)-b \bfE(\vphi), \text{ where } b:=b_{[\chi]}=\frac{n [\chi]\cdot [\omega]^{\cdot n-1}}{[\omega]^{\cdot n}}
\end{equation}
It is known by \cite{SW08, CS17} that if the class of $\chi$ satisfies the inequality:
\begin{equation}
[\chi]>0, \quad b [\omega]-(n-1)[\chi]>0
\end{equation}
then $\bfJ^\chi(\vphi)$ is proper, in particular bounded from below. 

Letting $[\chi]=-c_1(X)+\delta [\omega]$, then $b=-\ud{S}+n\delta $. For $\delta \gg 1$, $[\chi]>0$ and
\begin{eqnarray*}
b [\omega]-(n-1)[\chi]&=&(-\ud{S}+n\delta)[\omega]-(n-1)(-c_1(X)+\delta[\omega])\\
&=&(\delta-\ud{S}) [\omega]+(n-1)c_1(X)>0.
\end{eqnarray*}
In this case, we get $\cJ+\delta(\bfI-\bfJ)=\cJ^{-Ric(\Omega)+\delta \omega}\ge -C$. 

On the other hand, letting $[\chi]=\delta [\omega]+c_1(X)$, then $b=\ud{S}+n \delta$. For $\delta\gg 1$, $[\chi]>0$ and 
\begin{eqnarray*}
b [\omega]-(n-1)[\chi]&=&(\ud{S}+n\delta)[\omega]-(n-1)(c_1(X)+\delta[\omega])\\
&=&(\delta+\ud{S}) [\omega]-(n-1)c_1(X)>0.
\end{eqnarray*}
So in this case, we get $-\cJ+\delta (\bfI-\bfJ))=\cJ^{Ric(\Omega)+\delta \omega}$ is bounded from below. 

It is then easy to see that the wanted inequality easily follows.
\end{proof}

Recall that by Lemma \ref{lem-J2Kst}, $\cJ^{K_X}$-semistability implies uniform $K$-stability.
Similar argument as above proves Theorem \ref{thm-cJstable}, which says that  $\cJ^{K_X}$-semistability implies the existence of cscK metric. We indeed get a slightly stronger result as follows. 
\begin{thm}\label{thm-cJstable2}
If there exists $\delta<\alpha(X, L)$ such that 
\begin{equation}\label{eq-Jwst}
\cJ^\NA(\phi)\ge -\delta \bfI^\NA \text{ for all } \phi\in \mcH^\NA(L),
\end{equation}
then there exists a cscK metric on $(X, L)$.
\end{thm}

\begin{proof}
By the work of Chen-Cheng \cite{CC18a}, we just need to show that $\bfM$ is coercive.
Assume that $\bfM$ is not coercive. Then there exists a destabilizing geodesic ray $\Phi$ as in the proof of Proposition \ref{prop-exist1}. By using Jensen's inequality and Tian's $\alpha$-invariant (see \cite{Tia00} or the proof of Lemma \ref{lem-J2Kst}), we know that:
\begin{equation}
0\ge \bfM'^\infty(\Phi)\ge \alpha \bfI'^\infty(\Phi)+ \mcJ'^\infty(\Phi).
\end{equation}
Recall that $\cJ=\bfE^{-Ric(\Omega)}+\ud{S}\bfE$ in \eqref{eq-cJ}.
By the maximality of $\Phi$ (Theorem \ref{thm-maximal}) and convergence result in Theorem \ref{thm-Echislope}, for any $\{\phi_m\}\subset \mcH^\NA(L)$ decreasing to $\Phi_\NA$, we have
\begin{equation}
0\ge \lim_{m\rightarrow+\infty}(\mcJ^\NA(\phi_m)+\alpha \bfI^\NA(\phi_m)).
\end{equation}
But this contradicts the assumption \eqref{eq-Jwst} if $\delta<\alpha$, taking into account that $\lim_{m\rightarrow+\infty} \bfI^\NA(\phi_m)=\bfI'^\infty(\Phi)>0$ by \eqref{eq-Jslopeuni}.
\end{proof}

Regarding the condition \eqref{eq-Jwst}, we have the K\"{a}hler-Einstein cases and the cases essentially given by Dervan (\cite{Der16}) and Li-Shi (\cite{LS15}):
\begin{prop}
The condition \eqref{eq-Jwst} holds true in the following situations:
\begin{enumerate}
\item[(i)] (\cite{Oda12, BHJ17}) $K_X\equiv \lambda L$ with $\lambda\le 0$, and $\delta=\frac{\lambda}{n+1}\le 0$.
\item[(i)] (\cite{Oda12, BHJ17}) $K_X\equiv \lambda L$ with $\lambda>0$,  and $\delta=\lambda \frac{n}{n+1}$.

\item[(ii)] (\cite{Der16a}) $-(K_X+\frac{\ud{S}}{n+1}L)$ is nef, $\delta=\frac{\ud{S}}{n+1}+\epsilon$ for any $\epsilon>0$.

\item[(iii)] (\cite{LS15}) There exists $\epsilon>0$ such that both $K_X+\epsilon L$ and $(-\ud{S}+\epsilon)L-(n-1)K_X$ are ample, $\delta=\epsilon \frac{n}{n+1}$.

\end{enumerate}
\end{prop}
Since it is not our purpose to get the most general sufficient conditions, we just sketch the proof.
\begin{proof}[Sketch of proof]
For convenience of notations, we will use the following notation generalizing $(\cJ^{K_X})^\NA$. For any fixed $\bR$-line bundle $Q$ over $X$, we set $c_Q=\frac{n Q\cdot L^{\cdot n-1}}{L^{\cdot n}}$ and
\begin{equation}
(\cJ^Q)^\NA(\mcX, \mcL)=\qV Q\cdot \bar{\mcL}^{\cdot n}-\frac{c_Q}{n+1}\bar{\mcL}^{\cdot n+1}.
\end{equation}
Then we have the identity
\begin{equation}
(\bfI-\bfJ)^\NA=\qV \bar{\mcL}^{\cdot n}\cdot L_{\bP^1}-\frac{n}{n+1} \bar{\mcL}^{\cdot n+1}=(\cJ^L)^\NA(\mcX, \mcL).
\end{equation}
For the first two cases, we write:
\begin{eqnarray*}
\cJ^{\NA}&=&\qV  K_X\cdot \bar{\mcL}^{\cdot n}+\frac{\ud{S}}{n+1} \bar{\mcL}^{\cdot n+1}\\
&=&\frac{\ud{S}}{n}\left(\frac{n}{n+1}\bar{\mcL}^{\cdot n+1}-L\cdot \bar{\mcL}^{\cdot n}\right)+(\frac{\ud{S}}{n}L+K_X)\cdot \bar{\mcL}^{\cdot n}\\
&=&-\frac{\ud{S}}{n} (\bfI-\bfJ)^\NA+\bar{\mcL}^{\cdot n}\cdot \left(\frac{\ud{S}}{n}L+K_X\right).
\end{eqnarray*}
If $K_X\equiv \lambda L$, then $\ud{S}=n\lambda$ and we get $\cJ^\NA=-\lambda (\bfI-\bfJ)^\NA$. The first two cases follows from the inequality:
\begin{equation}
\frac{1}{n+1} \bfI^\NA \le (\bfI-\bfJ)^\NA\le \frac{n}{n+1}\bfI^\NA
\end{equation}
To explain Dervan's condition, first recall that
\begin{equation}
\bfI^\NA(\phi):=\qV \bar{\mcL}\cdot L_{\bP^1}^{\cdot n}-\frac{1}{n}\bar{\mcL}^{\cdot n+1}+\qV \bar{\mcL}^{\cdot n}\cdot L_{\bP^1}
\end{equation}
Next we set $\delta=\frac{\ud{S}+n\epsilon}{n+1}$ and calculate:
\begin{eqnarray}
\cJ^\NA+\delta\bfI^\NA&=&K_X\cdot \bar{\mcL}^{\cdot n}+\frac{\ud{S}}{n+1} \bar{\mcL}^{\cdot n+1}+\frac{\ud{S}+n\epsilon}{n+1}\left(\bar{\mcL}^{\cdot n}\cdot L+\bar{\mcL}\cdot L^{\cdot n}-\bar{\mcL}^{\cdot n+1}\right)\nonumber \\
&=&(K_X+\frac{\ud{S}-\epsilon}{n+1}L)\cdot \bar{\mcL}^{\cdot n}+\frac{\ud{S}+n\epsilon}{n+1}\bar{\mcL}\cdot L^n+\epsilon(L\cdot \bar{\mcL}^n-\frac{n}{n+1} \bar{\mcL}^{\cdot n+1})\nonumber \\
&\ge& -(\bfE^{Q})^\NA+\frac{\ud{S}+n\epsilon}{n+1}\sup(\phi-\phi_\triv) \label{eq-Dervest}
\end{eqnarray}
where $Q=-(K_X+\frac{\ud{S}}{n+1}L)+\frac{\epsilon}{n+1}L$ is ample for any $\epsilon>0$. By the argument of Dervan, we know that the right-hand-side of \eqref{eq-Dervest} is non-negative. So we get $\mcJ^\NA\ge -\delta \bfI^\NA$ for any $\delta>\frac{\ud{S}}{n+1}$ which implies the statement.

For the last case, we note that $(\cJ^Q)^\NA$ is linear in $Q$ so that
\begin{eqnarray*}
(\cJ^{K_X})^\NA&=&-\epsilon (\cJ^{L})^\NA+(\cJ^{K_X+\epsilon L})^\NA\\
&=&-\epsilon (\bfI-\bfJ)^\NA+(\cJ^{K_X+\epsilon L})^\NA.
\end{eqnarray*}
It is a result by Song-Weinkove \cite{SW08} that if the two conditions are satisfied, the corresponding $J$-equation is solvable, and by \cite{CS17} the Archimedean $\cJ^{K_X+\epsilon L}$-energy is proper, which implies the last statement.

\end{proof}


\begin{rem}
Gao Chen claimed that uniform slope-$\mathcal{J}^{K_X}$-stability (as defined by Lejmi-Sz\'{e}kelyhidi)
implies properness of $\mathcal{J}^{-Ric(\Omega)}$-energy,  which together with Chen-Cheng's result would also imply the the existence of cscK metric. 
\end{rem}

\begin{rem}\label{rem-XuOda}
In the first version of this paper, the author introduced some stability condition called $\tilde{K}$-stability. However as later pointed to me by C. Xu, Y. Odaka and Y. Gongyo \cite{OdG20}, this condition coincides with the $\mcJ^{K_X}$-stability and hence does not give us more information. 
\end{rem}

\vskip 3mm
\noindent
Department of Mathematics, Purdue University, West Lafayette, IN, 47907-2067.

\noindent
{\it Current address:} Department of Mathematics, Rutgers University, Piscataway, NJ 08854-8019.

\noindent
{\it E-mail address:} chi.li@rutgers.edu

\end{document}